\documentclass[12pt]{amsart}
\usepackage{amsmath}
\usepackage{amssymb}
\usepackage{mathrsfs}
\usepackage{verbatim}
\usepackage{enumitem}
\usepackage{svn}
\usepackage{stmaryrd}
\usepackage[all]{xy}
\usepackage{amsfonts}
\usepackage{comment}
\usepackage{tabu} %for math mode of tables
 \usepackage{geometry}
\usepackage[table]{xcolor}
\usepackage{adjustbox}
\usepackage{color}
\definecolor{ghcolor}{RGB}{0, 150, 200} %defines colors if you need in hyperref
\definecolor{winestain}{rgb}{0.5,0,0}
\usepackage[linktocpage=true, hidelinks, colorlinks, linkcolor=blue, citecolor=red]{hyperref} %linktocpage makes pages in toc clickable

\swapnumbers

\newtheorem{theorem}[subsubsection]{Theorem}
\newtheorem{thm}[subsubsection]{Theorem}
\newtheorem{lemma}[subsubsection]{Lemma}

\newtheorem{cor}[subsubsection]{Corollary}
\newtheorem{corollary}[subsubsection]{Corollary}

\newtheorem{prop}[subsubsection]{Proposition}
\newtheorem{co}[subsubsection]{Corollary}

\theoremstyle{definition}
\newtheorem{defn}[subsubsection]{Definition}
\newtheorem{definition}[subsubsection]{Definition}
\newtheorem{example}[subsubsection]{Example}

\newtheorem{convention}[subsubsection]{Convention}

\theoremstyle{remark}
\newtheorem{remark}[subsubsection]{Remark}

\numberwithin{equation}{subsection}

\def\numequation{\addtocounter{subsubsection}{1}\begin{equation}}
\def\nummultline{\addtocounter{subsubsection}{1}\begin{multline}}

\def \m {\mathfrak M}

 \def \E{\mathcal E}

\def \Z {\mathbb Z}
\def \inj {\hookrightarrow }
\def \to {\rightarrow}
\def \onto {\twoheadrightarrow}
\def \spec \text{spec}
\def \cont \text{cont}
 \def \M{\mathfrak M}

\def \GL {\t{GL}}
\def \L {\mathfrak L}

\def \Hom {\textnormal{Hom}}

\def  \cris {\textnormal{cris}}
\DeclareMathOperator{\gal}{Gal}

\DeclareMathOperator{\Fil}{Fil}

\def \Q {\mathbb Q}

\def \t {\textnormal}
\def \Z {\mathbb Z}

\def \ito {\overset  \sim  \to}
\def \O {\mathcal O}

\def \BF {\mathbb F}
\def \gs {\mathfrak S}

\def \ur {\t{ur}}

\def \N {\mathfrak N}

\def \m {\mathfrak M}
\def \Ker {\textnormal{Ker}}

\def \ku {k \llbracket u\rrbracket}
\def \v {\vee}

\def \Fr {\t{Fr}}
\def \Fil {\t{Fil}}

\def \W {\mathfrak W}

\def \gt {\mathfrak t}

\def \cW {\mathcal W}
\def \cT {\mathcal T}

\def \acris {{A_{\t{cris}}}}

\def \R {\mathcal R}

\def \sfi {\t{Mod}_{\gs}^{\varphi, r}}

\def \st {\textnormal{st}}

\def \upi {\underline \pi }

\def \< {\left <}
\def \> {\right >}

\def \hR {{\widehat \R} }
\def \hM {{\hat \M}}

\def \gr{\textnormal{gr}}
\def \upi {{\underline{\pi}}}

\def \inv {{\t{inv}}}

\def \Md {{\rm M}_d}

\def \fe {{\mathfrak e}}
\def \e {{\mathfrak e}}

\def \gZ {\mathfrak Z}

\def \Vbar  {\overline V}

\def \Ubar  {\overline U}

\def \cO { \mathcal O}
\def \Zp { \mathbb Z_p}
\def \Qp { \mathbb Q_p}
\def \Res {\textnormal{Res}}
\def \res {\textnormal{res}}

\def \Cor  {\textnormal{cor}}
\def \cor  {\textnormal{cor}}

\def \Qpbar {\overline{\Q_p}}

\def \Fpbar {\overline {\BF_p} }
\def \Fp  {\BF_p}

\def \FrR {\t{Fr} R}

\def \sto {\twoheadrightarrow}

\DeclareMathOperator{\Gal}{Gal}

\DeclareMathOperator{\Ind}{Ind}

\DeclareMathOperator{\rank}{rank}

\newcommand*{\wt}[1]{\widetilde{#1}}

\newcommand{\otimesvarphi}{\otimes_{\varphi}}

\begin{document}

\title{Loose crystalline lifts and overconvergence of \'etale $(\varphi, \tau)$-modules}

\author{Hui Gao}
\address{Department of Mathematics and Statistics, University of Helsinki, FI-00014, Finland}
\email{hui.gao@helsinki.fi}
\thanks{The first author is partially supported by China Postdoctoral Science Foundation Grant 2014M550539, and a postdoctoral position funded by Academy of Finland through Kari Vilonen.}

\author{Tong Liu}
\address{Department of Mathematics, Purdue University, West Lafayette, IN 47907}
\email{tongliu@math.purdue.edu}
\thanks{The second author is partially supported by NSF grant DMS-1406926.}

%    General info


\begin{abstract} Let $p$ be a prime, $K$ a finite extension of $\Qp$, and let $G_K$ be the absolute Galois group of $K$.
The category of \'etale $(\varphi, \tau)$-modules is equivalent to the category of $p$-adic Galois representations of $G_K$. In this paper, we show that all \'etale $(\varphi, \tau)$-modules are overconvergent; this answers a question of Caruso. Our result is an analogue of the classical overconvergence result of Cherbonnier and Colmez in the setting of \'etale $(\varphi, \Gamma)$-modules. However, our method is completely different from theirs.
Indeed, we first show that all $p$-power-torsion representations admit loose crystalline lifts; this allows us to construct certain Kisin models in these torsion representations.
We study the structure of these Kisin models, and use them to build an overconvergence basis.
\end{abstract}

\maketitle

\tableofcontents

\section{Introduction}

\subsection{Overview and main theorem}
Let  $p$ be a prime, $K$ a finite extension of $\Q_p$.
Let $\O_K$ be the ring of integers, $k$ the residue field, and $K_0$ the maximal unramified subextension. Then the ring of Witt vectors $W:=W(k)$ is just the ring of integers of $K_0$ and $K/ K_0$ is totally ramified. Write $e := [K: K_0]$, $f := [k : \BF_p]$ with $\BF_p := \Z/p\Z$. We fix an algebraic closure $\overline {K}$ of $K$ and set $G_K:=\Gal(\overline{K}/K)$.

Let $\rho : G_K \to \GL_d (\Z_p)$ be a continuous representation and let $T$ be the ambient space. In \cite{Car2}, for each such $T$ is associated an \'etale $(\varphi, \tau)$-module, which is a triple $\hat M = (M, \varphi_M, \hat G)$ (actually, we are using a variant of the definition in \cite{Car2}, see our Definition \ref{defn phi tau mod}). By \cite{Car2}, the category of finite free $\Zp$-representation of $G_K$ is equivalent to the category of finite free \'etale $(\varphi, \tau)$-modules. The functor from $T$ to $\hat M$ is via:
$$ M(T)  = \left ( \O_{\widehat \E^\ur} \otimes_{\Z_p} T \right ) ^{G_\infty} \ \text{ and } \  \hat M(T) = \left (W(\FrR) \otimes_{\Z_p} T \right ) ^{H_\infty} .$$
See \S \ref{sec phi tau} for any unfamiliar terms and more details. Here $ \O_{\widehat \E^\ur}$ and $W(\FrR)$ are certain ``period rings".

For the $\Zp$-representation $T$, we can also define its ``overconvergent periods" via:
$$ M^{\dagger, r} (T): = \left ( \O ^{\dag, r}_{\widehat  \E ^\ur} \otimes_{\Z_p} T \right ) ^{G_\infty} \ \text{ and } \  \hat M^{\dagger, r}(T) := \left (W(\FrR)^{\dag, r} \otimes_{\Z_p} T \right ) ^{H_\infty} ,$$
where $r \in \mathbb R^{>0}$, and $\O ^{\dag, r}_{\widehat  \E ^\ur}$ and $W(\FrR)^{\dag, r}$ are the ``overconvergent period rings".
We say that $\hat M = (M, \varphi_M, \hat G)$ is \emph{overconvergent} if it can be recovered by its ``overconvergent periods", \emph{i.e.}, if for some $r \in \mathbb R^{>0}$, we have
$$ M(T) = \O_\E  \otimes_{\O_\E ^{\dag, r}}   M^{\dagger, r} (T) \ \text{ and } \
 \hat M(T) = W(F_\tau) \otimes_{W(F_\tau) ^{\dag, r} } \hat M^{\dagger, r}(T).
$$
Our main theorem is the following, which answers the question of Caruso \cite[\S4]{Car2}:
\begin{thm}
For any finite free $\Z_p$-representation $T$ of $G_K$, its associated $(\varphi, \tau)$-module is overconvergent.
\end{thm}

\begin{remark} \label{rem: intro}
\begin{enumerate}
  \item In the $(\varphi, \Gamma)$-module setting, the corresponding overconvergence theorem in \cite{Colmez-overcon} (see also \cite{BCfamily, Kednew}) is valid without assuming the residue field of $K$ is finite (\emph{i.e.}, they only need to assume $k$ perfect). In this paper, finiteness of $k$ is needed for our loose crystalline lifting theorems; in particular, we will need to consider nontrivial finite unramified extensions of $K$ (cf. \S \ref{sec-lift}), which are not well-defined if $k$ is infinite.

   \item  By using the ideas in the current paper, we can give a reproof of the overconvergence of \'etale $(\varphi, \Gamma)$-modules (when $k$ is finite); see upcoming \cite{GL16phiGamma}.

  \item In \cite[\S 6,2]{KL2}, Kedlaya and R. Liu formulated another approach to the overconvergence question for $(\varphi, \tau)$-modules. The approach does not assume the finiteness of $k$. As they pointed out in \cite[Rem. 6.2.8]{KL2}, it would be very interesting to compare their approach with ours.

  \item In Berger's recent work \cite{Ber-mulivariable} proving certain overconvergence result in the Lubin-Tate setting, he critically used the calculation (in the $F$-analytic setting) of ``locally analytic vectors" in \cite{BC16}. In fact, in \emph{loc. cit.}, the ``locally analytic vectors" in the $(\varphi, \tau)$-module setting were also calculated (see \cite[\S 4.4]{BC16}). It would be interesting to see if the approach of \cite{BC16, Ber-mulivariable} can also be used to prove the overconvergence of $(\varphi, \tau)$-modules.
\end{enumerate}
\end{remark}

\subsection{Why overconvergence?}
Our theorem is clearly an analogue of the result of Cherbonnier and Colmez \cite{Colmez-overcon}, where they show that the $(\varphi, \Gamma)$-module associated to $T$ is overconvergent. Note that, the category of \'etale $(\varphi, \Gamma)$-modules, similarly as the category of \'etale $(\varphi, \tau)$-modules, is also equivalent to the category of all $p$-adic $G_K$-representations. Indeed, this similarity is what prompts Caruso to ask the question in \cite[\S4]{Car2}. The overconvergence result of \cite{Colmez-overcon} plays a fundamental role in the application of $(\varphi, \Gamma)$-modules to study various problems, e.g., the proof of $p$-adic local monodromy theorem (\cite{Ber0}), and the study of $p$-adic local Langlands correspondence for $\GL_2(\Qp)$(cf. \cite{CDP} for the most recent advance), to name just a few. To study $p$-adic local Langlands correspondence for $\GL_2(F)$ where $F/\Qp$ is a finite extension, there has been very recent work to generalize \cite{Colmez-overcon} to the Lubin-Tate setting, culminating in Berger's proof (see \cite{Ber-mulivariable}, also \cite{FX}).

Considering the very useful applications of overconvergent $(\varphi, \Gamma)$-modules, it is natural to expect that our overconvergent $(\varphi, \tau)$-modules should also generate interesting applications.
In particular, one of the reasons that the overconvergence result of \cite{Colmez-overcon} is important, is because that it allows us to link $(\varphi, \Gamma)$-modules to Fontaine modules (see \cite{Ber0}). Namely, overconvergence helps to link the category of \emph{all} Galois representations to the category of \emph{geometric} (\emph{i.e.}, semi-stable, crystalline) representations. However, in our current paper, we already use such a link to prove our overconvergence result. Namely, we use crystalline representations to ``approximate" general Galois representations (see \S \ref{subsec strategy} for a sketch). We believe that our approach will shed some more light on the deeper meaning of overconvergence, and in particular, our overconvergent $(\varphi, \tau)$-modules could potentially have certain advantages (than the overconvergent $(\varphi, \Gamma)$-modules) in applications.
We will report some of our progress in future papers.

\subsection{Strategy of proof} \label{subsec strategy}
It turns out that our method is very different from that used in \cite{Colmez-overcon}. We start with a conjecture of crystalline lifts which predicts that every mod $p$ residual representation admits a crystalline lift (with non-positive Hodge-Tate weights). We can not fully prove this conjecture. But we prove a weaker version (see Theorem \ref{thm-potential-lift}) which states that for each  mod $p$ residual representation, after restricting to a finite unramified extension, the residual representation does admit a crystalline lift.
This allows us to construct (inductively on $n$) loose crystalline lifts $\wt L_n$ of $T_n:=T/p^nT$. In fact, if we let $-h_n$ denote the minimal Hodge-Tate weight of $\wt L_n$, then $h_{n}=h_{n-1}+s$ where $s$ is a constant positive number independent of $n$, \emph{i.e.}, the growth of $h_n$ is \emph{linear}!
We will see that the increasing rate of $h_n$ is closely related to the overconvergence radius; indeed, it makes it possible to study the relation between (maximal) Kisin models inside $M_n : = M/ p^n M$ for different $n$.
After carefully analyzing the structure of such Kisin models, we are able to construct an overconvergent basis inside the  $(\varphi, \tau)$-module, and use it to prove our final overconvergence theorem.

\subsection{Structure of the paper} In \S \ref{sec phi tau}, we collect some basic facts about various categories of modules (in particular, \'etale $(\varphi, \tau)$-modules). We define what it means for an \'etale $(\varphi, \tau)$-module to be overconvergent, and state our main theorem.
In \S \ref{sec-lift}, we show that all $p$-power-torsion representations admit \emph{loose} crystalline lifts.
In \S \ref{sec max kisin}, we study maximal liftable Kisin models, and in particular show that they are invariant under finite unramified base change.
In \S \ref{sec torsion thy}, we study Kisin models in \'etale $\varphi$-modules corresponding to $p^n$-torsion representation of $G_K$.
Finally in \S \ref{sec OC}, we prove our main theorem.

\subsection{Notations}

\subsubsection{Some notations in $p$-adic Hodge theory}
We put $R:=\varprojlim \limits_{x\to x^p} \O_{\overline K}/ p \O_{\overline K}$, equipped with its natural coordinate-wise action of $G_K$.
 There is a unique surjective projection map $\theta :
W(R) \to \widehat \O_{\overline K}$ to the $p$-adic completion $ \widehat \O_{\overline K}$ of
$\O_{\overline K}$, which lifts the  projection $R \to \O_{\overline K}/ p$
onto the first factor in the inverse limit. We denote by $A_{\t {cris}}$ the $p$-adic completion
of the divided power envelope of $W(R)$ with respect to
$\t{Ker}(\theta)$.
%By the universal property of Witt vectors $W(R)$ of $R$, there
 As usual, we write $B_\t{cris}^+=
A_\t{cris}[1/p]$ and   $B_\t{dR}^+$ the $\t{Ker}(\theta)$-adic completion of $W(R)[1/p]$. For any subring $A \subset B^+_\t{dR}$, we define filtration on $A$ by  $\t{Fil}^i A = A \cap (\t{Ker}(\theta))^iB^+_\t{dR}$.

We fix a uniformizer  $\pi \in  \O_K$ and the Eisenstein polynomial $E(u) \in W[u]$ of $\pi$ in this paper. Define $\pi _n \in \overline K$ inductively such that $\pi_0 = \pi$ and $(\pi_{n+1})^p = \pi_n$. Define $\mu _n \in \overline K$ inductively such that $\mu_1$ is a primitive $p$-th root of unity and $(\mu_{n+1})^p = \mu_n$.
Set $K_{\infty} : = \cup _{n = 1} ^{\infty} K(\pi_n)$, $K_{p^\infty}=  \cup _{n=1}^\infty
K(\mu_{n})$,
and $\hat K:=  \cup_{n = 1} ^{\infty} K(\pi_n, \mu_n).$
Let $G_{\infty}:= \gal (\overline K / K_{\infty})$, $G_{p^\infty}:= \gal (\overline K / K_{p^\infty})$, $H_\infty: =\gal(\overline K/\hat K)$,
$H_{K}:= \gal (\hat K/ K_\infty)$, and $\hat G: =\gal (\hat K/K) $.
 Set  $\underline \varepsilon:= (\mu_{i}) _{i \geq 0} \in R$ and $t:= -\log([\underline \varepsilon])\in \acris$ as usual. We use $\varepsilon_p$ to denote the $p$-adic cyclotomic character.

Recall that $\{\pi_n\}_{n \geq 0}$ defines an element $\upi \in R$, and let $[\underline \pi ]\in W(R)$ be the Techm\"uller representative.
Recall that $\gs  = W[\![u]\!]$ with Frobenius extending the arithmetic Frobenius on $W(k)$ and $\varphi (u ) = u ^p$. %Many results here are taken from \cite{liu6-F-crys} but only for $F = \Q_p$.
We can embed the $W(k)$-algebra $W(k)[u]$ into
$W(R)\subset\acris$ by the map $u\mapsto [\underline \pi]$.   This
embedding extends to the  embedding $\gs \inj W(R)$ which is compatible with
Frobenious endomorphisms.

We denote  by $S$  the $p$-adic completion of the divided power
envelope of $W(k)[u]$ with respect to the ideal generated by $E(u)$. Write $S_{K_0}:= S[\frac{1}{p}]$.
There is a unique map (Frobenius) $\varphi_S: S \to S$ which extends
the Frobenius on $\gs$. %We write $N_S$ for the $K_0$-linear derivation on $S_{K_{0}}$ such that $N_S(u)= -u$.
Let  $\Fil ^n S\subset S $ be the $p$-adic completion of the ideal generated by $\gamma_i (E(u)):= \frac{E(u)^i}{i!}$ with $ i \geq n$. One can show that the embedding $W(k)[u] \to W(R)$ via $u \mapsto [\upi]$ extends to the embedding $S \inj A_\cris$ compatible with Frobenius $\varphi$ and filtration (note that $E([\upi])$ is a generator of $\Fil ^1 W(R)$).

As a subring of $\acris$, $S$ is not stable under the action of $G_K$,
though $S$ is fixed by $G_\infty$. Define a subring inside $B^+_\cris$:
$$\R_{K_0}: =\left\{x = \sum_{i=0 }^\infty f_i t^{\{i\}}, f_i \in S_{K_0} \t{
and } f_i \to 0\t{ as }i \to +\infty \right\}, $$ where $t ^{\{i\}}= \frac{t^i}{p^{\tilde q(i)}\tilde q(i)!}$ and $\tilde q(i)$ satisfies $i = \tilde q(i)(p-1) +r(i)$ with $0 \leq r(i )< p-1$. Define $\hR := W(R)\cap \R_{K_0}$. One
can show that $\R_{K_0}$ and $\hR$ are stable under the $G_K$-action and the $G_K$-action factors through  $\hat G$ (see \cite[\S 2.2]{liu4}). Let $I_+R$ be the maximal ideal of $R$ and $I_+ \hR = W(I_+R) \cap \hR$. By\cite[Lem. 2.2.1]{liu4}, one has
$\hR / I_+\hR\simeq \gs/ u \gs  = W(k)$.

\subsubsection{Fontaine modules and Hodge-Tate weights}
When $V$ is a semi-stable representation of $G_K$, we let $D_{\textnormal{st}} (V):  = (B_{\st} \otimes_{\Qp} V^{\vee})^{G_K}$ where $V^{\vee}$ is the dual representation of $V$. The Hodge-Tate weights of $V$ are defined to be $i \in \mathbb Z$ such that $\gr^i D_{\st}(V) \neq 0$.  For example, for the cyclotomic character $\varepsilon_p$, its Hodge-Tate weight is $\{ 1\}$.

Note here that we are using a contravariant Fontaine functor, which is the more usual convention in integral $p$-adic Hodge theory (e.g., as in \cite{liu2}). Later in our current paper, we will construct several covariant functors for our convenience in this paper.

\subsubsection{Some other notations}

Throughout this paper, we reserve $\varphi$ to denote Frobenius operator. We sometimes add subscripts to indicate on which object Frobenius is defined. For example, $\varphi_\M$ is the Frobenius defined on $\M$. We always drop these subscripts if no confusion will
arise. Let $A$ be a ring endowed with Frobenius $\varphi_A$ and  $M$  a module over $A$. We always denote $\varphi ^*M := A \otimes_{ \varphi_A, A } M$. Note that if $M$ has a $\varphi _A$-semi-linear endomorphism $\varphi_M: M \to M$ then $1 \otimes \varphi_M : \varphi ^* M \to M$ is an $A$-linear map. %We reserve $f(u) = u ^p + a_{p -1} u + \cdots+ a_1 u \equiv u ^p \mod p$ for $\varphi (u) = f(u)$.
We also reserve $v$ to denote valuations which are normalized so that $v(p)=1$.
Finally $\Md (A)$ always denotes the ring of $d \times d$-matrices with entries in $A$ and $I_d$ denotes the $d \times d $-identity matrix.

\subsection*{Acknowledgement} It is a great pleasure to thank Laurent Berger, Xavier Caruso, Kiran Kedlaya and Ruochuan Liu
for very useful conversations and correspondences. We would like to thank Toby Gee and David Savitt for allowing us to use results (mainly in our \S \ref{sec-lift}) from an unpublished draft \cite{GLS-unp}(which is collaborated work with the second author). We also thank the anonymous referee(s) for helping to improve the exposition.

\section{$(\varphi, \tau)$-modules and $(\varphi, \hat G )$-modules} \label{sec phi tau}
In this section, we first collect some basic facts on (integral and torsion) \'etale $\varphi$-modules, \'etale $(\varphi, \tau)$-modules, Kisin modules, $(\varphi, \hat G)$-modules and their attached representations. In particular, we also provide a covariant theory which will be more convenient for our paper. Then, we define what it means for an \'etale $(\varphi, \tau)$-module to be overconvergent, and state our main theorem.

\subsection{\'Etale $\varphi$-modules and $(\varphi, \tau)$-modules}

Recall that $\O _\E$ is the $p$-adic completion of $\gs[1/u]$.  %equipped with the unique continuous extension of $\varphi$.
Our fixed embedding $\gs\hookrightarrow W(R)$ determined by  $\upi$
uniquely extends to a $\varphi$-equivariant embedding $\iota:\O_{\E}\hookrightarrow W(\t{Fr} R)$ (here $\t{Fr} R$ denotes the fraction field of $R$), and we identify $\O_{\E}$ with its image in $W(\t{Fr} R)$.
We note that $\O_\E$ is a complete discrete valuation ring with uniformizer $p$ and residue field
$k (\!(\upi)\!)$ as a subfield of $\Fr R$. Let $\E$ denote the fractional field of $\O_\E$,  $\E ^\ur$ the maximal unramified extension of $\E$ inside $W(\t{Fr} R )[\frac 1 p ]$ and $\O_{\E ^\ur}$ the ring of integers. Set $\O_{\widehat{\E} ^\ur}$ the $p$-adic completion of $\O_{\E ^\ur}$ and  $\gs ^\ur : = W(R) \cap \O _{\widehat \E^\ur}$.

\begin{defn}
Let $'\t{Mod}_{\O_\E}^\varphi$ denote the category of finite type $\O_\E$-modules $M $  equipped with a $\varphi_{\O_\E}$-semi-linear endomorphism $\varphi _M : M\to M$ such that $1 \otimes \varphi : \varphi ^*M \to M $ is an isomorphism. Morphisms in this category  are just $\O_\E$-linear maps compatible with $\varphi$'s. We call objects in $'\t{Mod}_{\O_\E}^\varphi$ {\em \'etale $\varphi$-modules}.
\end{defn}

\subsubsection{} Let $'\t{Rep}_{\Z_p}(G_{\infty}) $ (resp. $'\t{Rep}_{\Z_p}(G_{K}) $ ) denote the category of finite type $\Z_p$-modules $T$   with a continuous $\Z_p$-linear $G_{\infty}$ (resp. $G_K$)-action.
For $M $ in $'\t{Mod}_{\O_\E}^\varphi$, define
$$ V(M):= ( \O _{\widehat \E ^\ur} \otimes_{\O_\E} M) ^{\varphi =1}.  $$
For $T $ in  $'\t{Rep}_{\Z_p}(G_{\infty}) $, define
$$ \underline M(T):= ( \O _{\widehat \E ^\ur} \otimes_{\Z_p} T) ^{G_\infty}. $$

\begin{thm}[{\cite[Prop. A 1.2.6]{fo4}}] \label{thmFon}
The functors $V$ and $\underline M$ induces an exact tensor equivalence between the categories $'\t{Mod}_{\O_\E} ^\varphi$  and $'\t{Rep}_{\Z_p} (G_{\infty})$.
\end{thm}

For later use, we also record the following lemma.
\begin{lemma} \label{Eq-V(M)}
Let $M$ be an \'etale $\varphi$-module, then
$$V(M) = ( M \otimes_{\O_\E} W(\t{Fr} R)) ^{\varphi =1}.$$
\end{lemma}
\begin{proof}
Set $V' (M) : = ( M \otimes_{\O_\E} W(\t{Fr} R)) ^{\varphi =1}$. It is obvious that $V(M) \subset  V'(M)$ as $\O_{\widehat \E^\ur} \subset W(\FrR)$. To show that $V(M) = V'(M)$ we easily reduce to the case that $M$ is killed by $p ^n$. By standard d\'evissage procedure, we can produce an exact sequence
$$ 0 \to M ' \to M \to M'' \to 0 $$
with $M'$ killed by $p ^{n -1}$  and $M''$ killed by $p$. By diagram chasing, to show that $V(M) \supset  V'(M)$, it suffices to prove the case when $n = 1$. Recall that $\O_{\widehat \E^\ur} / p = k(\!(u)\!) ^{\t{sep}}$ the separable closure of $k (\!(u)\!)$. It is well-known that (indeed, by \ref{thmFon})
$$V(M)\otimes_{\BF_p} k(\!(u)\!) ^{\t{sep}} \simeq M \otimes_{k (\!(u)\!)} k(\!(u)\!) ^{\t{sep}}  $$
compatible with $\varphi$-actions. So $V(M)\otimes_{\BF_p} \t{Fr}  R \simeq M \otimes_{k (\!(u)\!)}  \t{Fr} R $ compatible with $\varphi$-actions. Taking $\varphi$-invariants on both sides, since $(\t{Fr} R) ^{\varphi= 1 } = \BF_p$, we see that $V(M) = V'(M)$.
\end{proof}

Now let us recall the theory  of $(\varphi, \tau)$-modules as in \cite{Car2}. Let $H_\infty = \gal (\overline K / \hat K)$ and $F_\tau : = (\t{Fr} R)^{H_\infty}$. As a subring of $W(\FrR)$,  $W(F_\tau)$ is stable on $G_K$-action and the action factors through $\hat G$.

\begin{defn} \label{defn phi tau mod}
An \'etale \emph{$(\varphi, \tau)$-module} is a triple $(M, \varphi_M, \hat G)$ where
\begin{itemize}\item $(M , \varphi_M)$ is an \'etale $\varphi$-module;
\item $\hat G$ is a continuous $W(F_\tau)$-semi-linear $\hat G$-action on $\hat M : =W(F_\tau) \otimes_{\O_\E} M$, and $\hat G$ commutes with $\varphi_{\hat M}$ on $\hat M$, \emph{i.e.,} for
any $g \in \hat G$, $g \varphi_{\hat M} = \varphi_{\hat M} g$;
\item regarding $M$ as an $\O_\E $-submodule in $ \hat M $, then $M
\subset \hat M ^{H_{K}}$.
\end{itemize}
\end{defn}

\begin{remark}
Note that our definition of \'etale $(\varphi, \tau)$-modules is slightly different from \cite[Definition 1.7]{Car2} used by Caruso (also note that Caruso simply calls them $(\varphi, \tau)$-modules without the term ``\'etale").
When $\hat G \simeq G_{p ^\infty} \rtimes H_K$ (e.g. when $p \not =2$ assumed in \cite{Car2}), since $G_{p^\infty}\simeq \Z_p (1)$,  we can always pick a topological generator $\tau \in G_{p ^\infty}$. Then the last two axioms in the above definition can be reinterpreted as axioms for $\tau$ as in \cite[Definition 1.7]{Car2}.
In such situations, it is easy to show that our definition is equivalent to \textit{loc. cit.}.
\end{remark}

Given an \'etale $(\varphi, \tau)$-modules $\hat M = (M, \varphi_M, \hat G)$, we define
$$ \cT^* (\hat M):= ( W(\FrR ) \otimes _{\O_\E} M) ^{\varphi=1} = \left ( W(\FrR) \otimes_{W(F_\tau)} \hat M  \right )^{\varphi =1}.$$
Note that $W(\FrR) \otimes_{\O_\E} M$ has a $G_K$-action, which is induced from that on $\hat M = W(F_\tau) \otimes_{\O_\E} M$. Since $G_K$-action commutes with $\varphi$, $\cT^*(\hat M)$ is a representation of $G_K$.

\begin{prop}\label{prop-cT*}
Notations as the above. Then
\begin{enumerate} \item $\cT^*(\hat M)|_{G_\infty} \simeq V(M)$.
\item The functor $\cT^*$ induces an equivalence between the category of $(\varphi, \tau)$-modules and the category $'\t{Rep}_{\Z_p}(G_{K}) $.
\end{enumerate}
\end{prop}
\begin{proof}
(1) follows from Lemma \ref{Eq-V(M)} and the definition of \'etale $(\varphi, \tau)$-modules.
%and the assumption that $M \subset (\hat M) ^{G_\infty} $.

(2) For an \'etale $\varphi$-module $M$, we have a natural isomorphism $$ \O_ {\widehat \E^\ur} \otimes _{\O_\E} M \simeq \O_ {\widehat \E^\ur} \otimes_{\Z_p} V(M)$$ which is compatible with $G_\infty$-action and $\varphi$-action on both sides. Tensoring $W(\Fr R)$ on both sides and using (1),  we have
$$ W(\Fr R) \otimes _{\O_\E} M \simeq W (\Fr R) \otimes_{\Z_p} V(M) \simeq W(\FrR) \otimes_{\Z_p} \cT ^*(\hat M)$$
By the construction of $\cT^*(\hat M)$, it is clear the above isomorphism is compatible with $G_K$-actions on both sides, where the $G_K$-action on the left is induced from that on $\hat M$. In particular, we can recover $M$ and $\hat M$ via the formula: $$ M = \left ( \O_{\widehat \E^\ur} \otimes_{\Z_p}V(M) \right ) ^{G_\infty} \ \text{ and } \  \hat M = \left (W(\FrR) \otimes_{\Z_p} \cT ^*(\hat M) \right ) ^{H_\infty} $$
Then we can easily show that the functor $\cT^*$ is fully faithful and essentially surjective.
\end{proof}

\subsection{Kisin modules and $(\varphi, \hat G)$-modules}\label{subsect-etale}
\begin{defn} For a nonnegative integer $r$,
we write $'\sfi$  for the category of finite-type $\gs$-modules $\M$ equipped with
a $\varphi_{\gs}$-semilinear endomorphism $\varphi_\M : \M \to \M$ satisfying
\begin{itemize}
	\item the cokernel of the linearization $1\otimes \varphi: \varphi ^*\M \to \M$ is killed by $E(u)^r$;
	\item the natural map $\M \to \O_\E \otimes_{\gs} \M$ is injective.
\end{itemize}
 Morphisms in $'\sfi$ are $\varphi$-compatible $\gs$-module homomorphisms.
\end{defn}
We call objects in $'\sfi$ \emph{Kisin module of $E(u)$-height $r$}. The category of {\em finite free Kisin modules of $E(u)$-height $r$},
denoted $\sfi$, is the full subcategory of $'\sfi$ consisting
of those objects which are finite free over $\gs$. We call an object $\M \in {'\sfi}$ a {\em torsion Kisin module of $E(u)$-height $r$} if $\M$ is killed by $p^n$ for some $n$. Since $E(u)$ is always fixed in this paper, we often drop $E(u)$ from the above notions.

For any finite free Kisin module $\M \in \t{Mod}_{\gs} ^{\varphi, r}$, we define $$T_\gs (\M) : = \Hom_{\gs, \varphi} (\M , W(R)).$$ See \cite[\S2.2]{liu2} for more details on $T_\gs$. In particular, $T_\gs (\M)$ is a finite free $\Z_p$-representation of $G_\infty$ and $\rank_{\Z_p} T_\gs(\M) = \rank_\gs (\M)$.

Now let us review the theory of $(\varphi, \hat G)$-modules.
\begin{defn}
Following \cite{liu4}, a finite free (resp. torsion) \emph{$(\varphi, \hat
G)$-module of height $ r$} is a triple $(\M , \varphi, \hat G)$ where
\begin{enumerate}
\item $(\M, \varphi_\M)\in {'\sfi}$ is a finite free (resp. torsion) Kisin module of height $ r$;
\item $\hat G$ is a continuous $\hR$-semi-linear $\hat G$-action on $\hat \M: =\hR
\otimes_{\varphi, \gs} \M$;
\item $\hat G$ commutes with $\varphi_{\hM}$ on $\hM$, \emph{i.e.,} for
any $g \in \hat G$, $g \varphi_{\hM} = \varphi_{\hM} g$;
\item regard $\M$ as a $\varphi(\gs)$-submodule in $ \hM $, then $\M
\subset \hM ^{H_{K}}$;
\item $\hat G$ acts on $W(k)$-module $M:= \hM/I_+\hR\hM\simeq \M/u\M$ trivially.
\end{enumerate}
 Morphisms between  $(\varphi, \hat G)$-modules are morphisms of Kisin modules that commute with $\hat G$-action  on $\hM$'s.
\end{defn}

\begin{remark}\label{rm-issues}
When $\M$ is a torsion or finite free Kisin module, then the natural map $\M \to \hM $ induced by $\M \mapsto 1\otimes_{\varphi, \gs} \M$ is injective (see \cite[Lem.3.1.2]{liu-car2} and discussion above the lemma). So it makes sense to regard  $\M$ as a $\varphi(\gs)$-submodule of $\hM$ (item (4) in above definition). Also, the topology on $\hM$ is induced by the (weak) topology of $W(R)$ which is compatible with the $p$-adic topology of $A_\cris$.
\end{remark}

\subsubsection{}
  For a finite free  $(\varphi, \hat G)$-module  $\hM=(\M, \varphi, \hat G)$,
  we can associate a $\Z_p[G_K]$-module:
\begin{equation}\label{hatT}
\hat T (\hM) := \t{Hom}_{\hR, \varphi}(\hR \otimes_{\varphi, \gs}
\M, W(R)),
\end{equation}
where $G_K$ acts on $\hat T(\hM)$ via $g (f)(x) = g (f(g^{-1}(x)))$
for any $g \in G_K$ and $f \in \hat T(\hM)$.

If $T$ is a finite free $\Z_p$-representation of $G_\infty$ or $G_K$ we denote $T^\v: = \Hom_{\Z_p} (T, \Z_p)$ the dual representation of $T$. We fix  $\gt \in W(R)$ so that $\varphi (\gt) = c_0 ^{-1}E(u) \gt$ and $\gt \not = 0 \mod p$, where $c_0=\frac{E(0)}{p}$. Such $\gt$ exists and is unique up to multiplying $\Z_p ^\times$ (see \cite[Example 2.3.5]{liu2}). With these notations, we have the following:

\begin{theorem}[\cite{liu4}]\label{thm-review}
\begin{enumerate}
\item  $\hat T$ induces an anti-equivalence between the category of finite free
$(\varphi, \hat G)$-modules of height $r$ and the category of
$G_K$-stable $\Z_p$-lattices in semi-stable representations of $G_K$
with Hodge-Tate weights in $\{0, \dots, r\}$.
\item $\hat T $ induces a natural $W(R)$-linear injection
\begin{equation}\label{Eq: iota}
\hat \iota: \  W(R)\otimes_{\varphi, \gs} \M  \longrightarrow \hat T^\v
(\hat \M) \otimes_{\Z_p} W(R),
\end{equation}
 which is  compatible with Frobenius and $G_K$-actions on both sides (here, we use $\hat T^\v
(\hat \M)$ to mean the dual of $\hat T^(\hat \M)$). Moreover,
 $$(\varphi(\gt))^r (\hat T^\v(\hM) \otimes
_{\Z_p}W(R)) \subset\hat \iota( W(R)\otimes_{\varphi, \gs} \M).$$
\item There exists a natural isomorphism $T_\gs(\M) \ito \hat T (\hM)$ of $\Z_p[G_\infty]$-modules.
\end{enumerate}
\end{theorem}

\subsection{Covariant theory}
Note that $T_\gs$ and $\hat T$ are contravariant functors which is not convenient for us in many ways. In this section, we construct covariant variants for $T_\gs$ and $\hat T$.

Let $\M\in {'\sfi}$ be a Kisin module of height $r$, we define
\begin{equation}\label{Eq-define TM}
T^*_\gs (\M) := \left ( \M \otimes_{\gs}W(\FrR)\right)^{\varphi =1}.
\end{equation}
%And then we define for $(\varphi, \hat G)$-modules similarly.
Since $\gs \subset W(R) ^{G_\infty}$, we see that $G_\infty$ acts on $T^*_\gs (\M)$. %  Recall $a_1$ is the coefficient of linear time of $f(u)$.

\begin{lemma}\label{lem-covariant}
There is a natural isomorphism  of $\Z_p[G_\infty]$-modules
$$T^*_\gs (\M) \simeq V(\O_\E \otimes_\gs \M) $$
\end{lemma}
\begin{proof}
Note that $M := \O_\E \otimes_{\gs} \M$ is an \'etale $\varphi$-module, and the lemma follows from Lemma \ref{Eq-V(M)}.
\end{proof}

Now let us turn to the situation of $(\varphi, \hat G)$-modules. Given $\hM = (\M, \varphi_\M , \hat G) $ a $(\varphi, \hat G)$-module, either finite free or torsion, we define
$$ \hat T ^* (\hM) : = (W(\FrR) \otimes_{\varphi, \gs} \M ) ^{\varphi =1}.$$
Note that $W(\FrR) \otimes_{\varphi, \gs} \M$ has a $G_K$-action induced from that on $\hM: = \hR \otimes_{\varphi, \gs}\M$. So $\hat T ^*(\hM)$ is indeed a $\Z_p [G_K]$-module.

Now we obtain a covariant version of Theorem \ref{thm-review}:

\begin{theorem}\label{thm-old-covariant}
\begin{enumerate}
\item  $\hat T^*$ induces an equivalence between the category of finite free
$(\varphi, \hat G)$-modules of height $r$ and the category of
$G_K$-stable $\Z_p$-lattices in semi-stable representations of $G_K$
with Hodge-Tate weights in $\{-r, \dots, 0 \}$.

\item If $\M$ is either finite free or torsion then there exists a natural isomorphism $T^*_\gs(\M) \ito \hat T^* (\hM)$ of $\Z_p[G_\infty]$-modules.

\end{enumerate}
\end{theorem}
\begin{proof} (1) It suffices to prove that $\hat T^*(\hM) $ is isomorphic to  $\hat T^\vee (\hM)$, the dual of $\hat T(\hM)$. Tensor $W(\FrR)$  on both sides of \eqref{Eq: iota}.  Since $\varphi (\gt)$ is a unit in $W(\FrR)$,   we get an isomorphism
$$ W(\FrR) \otimes_{\varphi,\gs} \M \simeq \hat T^\vee( \hM) \otimes_{\Z_p} W(\FrR)$$ which is compatible with $G_K$-actions and $\varphi$-actions on the both sides. By taking $\varphi$-invariants on  both sides, we get $\hat T^*(\M) \simeq \hat T ^\v (\hM)$.

(2) Since $\varphi : W(\FrR) \to W(\FrR)$ is a ring isomorphism,  $W(\FrR) \otimes _\gs \M \to W(\FrR ) \otimes_{\varphi, \gs} \M $ induced by $\varphi\otimes 1 $ is an isomorphism which is compatible with Frobenius and $ G_\infty$-actions. In particular, we have $ T^*_\gs (\M ) = (W(\FrR ) \otimes_{\varphi, \gs} \M )^{\varphi=1}$. Combining the fact that $\M$ is a $\varphi(\gs)$-submodule of $\hM^{H_K }$, we see that $\hat{T}^*(\hM ) \simeq T ^*_\gs(\M)$ as $\Z_p[G_\infty]$-modules.
\end{proof}

\begin{convention}
From now on throughout this paper, when we use semi-stable (indeed, most of the time, crystalline) representations, without further mentioning, we always mean representations with non-positive Hodge-Tate weights (so that we can attach $(\varphi, \hat G)$-modules or Kisin modules via Theorem \ref{thm-old-covariant}).
\end{convention}

\subsubsection{}
For any integer  $s \ge 0$, let $\M(\varepsilon_p ^{-s})$ denote the Kisin module corresponding to $\varepsilon_p^{-s}$ via our Theorem \ref{thm-old-covariant}.
By \cite[Example 2.3.5]{liu2} (and our Theorem \ref{thm-old-covariant}),
this is a rank-1 $\gs$-module with a base $\fe$ so that $\varphi (\fe) = (c_0 ^{-1} E(u))^s \fe $ where $c_0=E(0)/p$.  For any Kisin module $\M$, we denote $\M (s):  = \M \otimes \M (\varepsilon _p^{-s})$. So $\varphi_{\M({-s})} = (c_0^{-1} E(u))^s \varphi_\M$.

\subsection{Kisin models and $(\varphi, \hat G)$-models}

\begin{defn}[\cite{kisin2}]
\begin{enumerate}
  \item Given an \'etale $\varphi$-module $M$ in  $'\t{Mod} _{\O_\E} ^{\varphi}$.   If $\M \in {'\sfi}$ is a Kisin module so that $M =  \O_\E \otimes_\gs \M $ then $\M$ is called a \emph{Kisin model} of $M$, or simply a \emph{model} of $M$.

  \item  Given $\hat M:  = (M, \varphi_M, \hat G_M)$ a torsion (resp. finite free) $(\varphi, \tau)$-module. A torsion (resp. finite free) $(\varphi, \hat G)$-module  $\hM : = (\M, \varphi_\M, \hat G)$ is called a \emph{model} of $\hat M$ if $\M$ is a model of $M$ and the isomorphism
$$ W(F_\tau) \otimes_{\hR} \hM \simeq W(F_\tau) \otimes_{\varphi, W(F_\tau)} \hat M$$
induced by $\O_\E \otimes_{\gs} \M \simeq M $ is compatible with $\hat G$-actions on both sides.
%Here $\hM : = \hR \otimes_{\varphi, \gs}\M$ and $\hat M: = W(F_\tau) \otimes_{\O_\E} M$.
\end{enumerate}
\end{defn}

\begin{lemma}
  \begin{enumerate}
    \item If $\M$ is a model of $M$, then $T^*_\gs(\M)\simeq V(M)$ as  $\Z_p[G_\infty]$-modules.
    \item If $\hM$ is a model of $\hat M$ then $\hat T^*(\hM)\simeq \cT ^*(\hat M)$ as $\Z_p[G_K]$-modules.
  \end{enumerate}
\end{lemma}
\begin{proof}
(1) is Lemma \ref{lem-covariant}.

For (2), by the definition of $(\varphi, \hat G)$-model, we get an isomorphism
 $$ W(\FrR ) \otimes_{\varphi, \gs} \M  \simeq W(\FrR) \otimes_{\varphi, \O_\E} M $$
which is compatible with $\varphi$-actions and $G_K$-actions on both sides. By the proof of Theorem \ref{thm-old-covariant},
we see that  the isomorphism $W(\FrR) \otimes _  {\O_\E}M \to W(\FrR ) \otimes_{\varphi, \O_\E} M $ induced by $\varphi \otimes 1$ is also  compatible with Frobenius and $ G_\infty$-actions. By taking Frobenius invariants on these two isomorphisms, we see that  $\hat T^*(\hM)\simeq \cT ^*(\hat M)$.
\end{proof}

Clearly, the most natural models of \'etale $(\varphi, \tau)$-modules come from lattices in semi-stable representations.
\begin{lemma} \label{lem model}
Suppose $T$ is a $G_K$-stable $\Z_p$-lattice in a semi-stable representation  of $G_K$ with Hodge-Tate weights in $\{-r, \dots, 0 \}$.
Let $\hat M:  = (M, \varphi_M, \hat G_M)$ be the $(\varphi, \tau)$-module associated to $T$ via Proposition \ref{prop-cT*}.
Let $\hM : = (\M, \varphi_\M, \hat G)$ be the $(\varphi, \hat G)$-module associated to $T$ via Theorem \ref{thm-old-covariant}.
Then $\hM$ is a model of $\hat M$.
\end{lemma}
\begin{proof}
Easy.
\end{proof}

Now suppose that $T_n$ is a $p$-power torsion representation of $G_K$, and $M_n$ the \'etale $\varphi$-module associated to $T_n|_{G_\infty}$. A natural source of Kisin models of $M_n$ comes from the following. Suppose that we have a surjective map of $G_K$-representations $f: L \onto T_n$ where $L$ is a semi-stable finite free $\Zp$-representation, then it induces the surjective map (which we still denote by $f$) $f: \mathcal L \onto M_n$, where $\mathcal L$ is the \'etale $\varphi$-module associated to $L|_{G_\infty}$. If $\mathfrak L$ is the Kisin module associated to $L$, then by Lemma \ref{lem model} above, $\mathfrak L$ is a Kisin model of $\mathcal L$. And so $f(\mathfrak L)$ is clearly a Kisin model of $M_n$.

\begin{example} \label{ex-cyclotomic} Let $T_1= \BF_p $ be the trivial $G_K$-representation and $M_1$ denote the corresponding trivial \'etale $\varphi$-module. Then $\M_1 = \ku$ is a model of $M_1$ which is realized by the surjection $\Z_p \twoheadrightarrow \BF_p $.
We also have another surjection $\Z_p (1-p) \twoheadrightarrow  \BF_p$, which realizes another model $\M_1 (p-1)\subset M_1$.
It is easy to check that $\M_1 (p-1) = u ^{e} \M_1$.

\end{example}

\subsection{Overconvergence and the main result}
For any $x \in W(\Fr R)$, we can write $x = \sum \limits_{i = 0}^\infty p ^i [x _i]$ with $x_i\in \Fr R$. Denote $v_R (\cdot)$ the valuation on $\FrR$ normalized by $v_R(\upi) =1$. For any $r \in \mathbb R^{>0}$, set
$$ W(\Fr R) ^{\dag, r}: =\left \{ x= \sum _{i = 0}^\infty p ^i [x _i] \in W(\FrR) | \  i  +   r v_R (x_i ) \to + \infty \right \}. $$
It turns out that $ W(\FrR) ^{\dag, r}$ is a ring, stable under $G_K$-action but not Frobenius, namely $\varphi (W(\FrR)^{\dag, r}) = W(\FrR)^{\dag, r/p}$. See \cite[\S II.1]{Colmez-overcon} for more details.
For any subring $A \subset W(\Fr R)$, denote
$$A^{\dag, r}: =A \cap W(\Fr R) ^{\dag, r}.$$

%Let $\O ^{\dag, r}_{\widehat  \E ^\ur} := \O _{\widehat \E ^\ur} \cap W(\Fr R) ^{\dag, r}$ and $\O_\E ^{\dag, r} := W(\Fr R) ^{\dag,r} \cap \O_\E$.
Recall that for a finite free $\Z_p$-representation $T$ of $G_K$, we can associate the $(\varphi, \tau)$-module via:
$$ M(T)  = \left ( \O_{\widehat \E^\ur} \otimes_{\Z_p} T \right ) ^{G_\infty} \ \text{ and } \  \hat M(T) = \left (W(\FrR) \otimes_{\Z_p} T \right ) ^{H_\infty} .$$
Now for any $r>0$, we define
$$ M^{\dagger, r} (T): = \left ( \O ^{\dag, r}_{\widehat  \E ^\ur} \otimes_{\Z_p} T \right ) ^{G_\infty} \ \text{ and } \  \hat M^{\dagger, r}(T) := \left (W(\FrR)^{\dag, r} \otimes_{\Z_p} T \right ) ^{H_\infty} .$$

\begin{definition}\label{Def-OC}
For a finite free $\Z_p$-representation $T$ of $G_K$, its associated $(\varphi, \tau)$-module is called overconvergent if there exists $r > 0$ such that
$$ M(T) = \O_\E  \otimes_{\O_\E ^{\dag, r}}   M^{\dagger, r} (T) \ \text{ and } \
 \hat M(T) = W(F_\tau) \otimes_{W(F_\tau) ^{\dag, r} } \hat M^{\dagger, r}(T).
$$
\end{definition}

The main theorem of our paper is the following:
\begin{thm}
For any finite free $\Z_p$-representation $T$ of $G_K$, its associated \'etale $(\varphi, \tau)$-module is overconvergent.
\end{thm}

The main input to prove the above theorem is the following, which says that there exists an ``overconvergent basis", with respect to which all entries of matrices for $\varphi$ and $\hat G$ are overconvergent elements. In fact, we can make these matrix entries to fall in more precise rings.

\begin{theorem}\label{thm-main2}
Let $\rho : G_K \to \GL_d (\Z_p)$ be a continuous representation and $\hat M = (M, \varphi_M, \hat G)$ the associated \'etale $(\varphi, \tau)$-module.
Then there exists an $\O_\E$-basis of $M$, and a constant $\alpha = \alpha (p, f, e , d)$ depending only on $p, f, e$ and $d$, such that with respect to this basis,
\begin{itemize}
  \item the matrix of $\varphi_M$ is in $\Md(\gs [\![\frac{p}{u ^\alpha}]\!])$,
  \item the matrix of $\tau$ is in $\Md(W(R)[\![\frac{p}{u ^\alpha} ]\!])$ for any $\tau \in \hat G$.
\end{itemize}
\end{theorem}

\section{Loose crystalline lifts of torsion representations}\label{sec-lift}
In this section, we show that all $p$-power-torsion representations admit loose crystalline lifts. The proof is via a inductive method. We first show that for a mod $p$ representation of $G_K$, if we restrict it to a certain subgroup $G_{K'}$ where $K'/K$ is finite unramified, then it admits strict crystalline lift.

\subsection{Definition of strict and loose lifts}
Let $E$ be a finite extension of $\Q_p$, $\O_E$ its ring of integers, $\varpi_E$ a uniformizer, $\mathfrak m_E$ the maximal ideal and $k_E= \O_E/ \mathfrak m_E$ the residue field.

\begin{defn}\label{defn loose lift}
\begin{enumerate}
  \item For $\bar \rho: G_K \to \GL_d(k_E)$ a torsion representation, we say that a continuous representation $r : G_K \to \GL _d (\O_E)$ is a \emph{strict} $\O_E$-lift of $\bar \rho$ if $r(\bmod \mathfrak m_E) \simeq \bar \rho$.

  \item  Suppose that $\bar \rho$ is a finite $\O_E$-module with a continuous (discrete) $G_K$-action and  killed by some $p^n$. Then we say that $\bar \rho$ admits a \textbf{loose} $\O_E$-lift $r $ if $r$ is a finite free $\O_E$-representation of $G_K$ and there exists a \emph{surjective} morphism $f : r \onto \bar \rho$ as $\O_E[G_K]$-modules.
\end{enumerate}
\end{defn}

\begin{definition}\label{convention: big enough E}
For a $p$-torsion representation of the form $\bar \rho: G_K \to \GL_d(k_E)$, we say that ``$E$ is big enough" (for $\bar\rho$) if each direct summand of $\bar \rho^{\textnormal{ss}}$ has strict crystalline $\O_E$-lift with non-positive Hodge-Tate weights.
\end{definition}

\begin{lemma}\label{rmk: big E}
\begin{enumerate}
  \item Suppose $\bar \rho: G_K \to \GL_d(\Fp)$ is absolutely irreducible (\emph{i.e.}, $\bar\rho\otimes_{\Fp}\Fpbar$ is irreducible), then $\bar\rho \otimes_{\Fp} k_d$ admits crystalline strict $W(k_d)$-lifts with Hodge-Tate weights in $[-(p^{fd}-2), 0]$. Here, $k_d$ is the degree $d$ extension of $k$, and $W(k_d)$ is the ring of Witt vectors.
  \item Let $\bar \rho: G_K \to \GL_d(k_E)$ be a $p$-torsion representation. If $E$ contains $K_d:=W(k_d)[1/p]$, then $E$ is big enough for $\bar\rho$.
\end{enumerate}
\end{lemma}
\begin{proof}
Both statements follow from \cite[Prop. 2.1.2]{Muller} (note that there is a minor error in the statement of \emph{loc. cit.}, namely, his $\bar \rho$ should have image in $\GL_d(k_d)$, not $\GL_d(k)$).
Indeed, for (2), without loss of generality, we can assume that $\bar \rho: G_K \to \GL_d(k_E)$ is absolutely irreducible. Then in the notation of \emph{loc. cit.}, the tame character $\omega_{\varpi}^r$ always lands in $k_d$ (before embedding to $\overline \Fp$), and so admits strict crystalline $W(k_d)$-lifts (and thus $\O_E$-lifts); the unramified character $\bar \mu: G_{K_d}\to k_E^{\times}$ certainly can be lifted.
\end{proof}

\begin{remark}  \label{rem: new rem E big}
  So in our Definition \ref{convention: big enough E}, any $E$ that contains $K_d:=W(k_d)[1/p]$ is big enough. However, in our writing, we will stick with $E$ and $k_E$ (rather than $K_d$ and $k_d$) with an eye for future application. In particular, in future applications, we might need to have $E$ to contain the Galois closure of $K_d$.
\end{remark}

\begin{remark}
\begin{enumerate}
 \item For $\bar \rho: G_K \to \GL_d(k_E)$ where $E$ is big enough, we tend to believe that strict crystalline $\O_E$-lift always exists. See the introduction of \cite{GHLS}. In particular, we know this is true for $d \le 3$ by \cite[Prop. 2.5.7]{Muller}.

  \item For $\bar \rho: G_K \to \GL_d(\Fp)$, we can certainly consider strict $\Zp$-lifts, but in general, even if  $\bar \rho$ is irreducible (or even absolutely irreducible), we do not know if such strict (crystalline) $\Zp$-lifts exist at all.

  \item Suppose that $\bar \rho$ is a torsion $\Zp$-representation, so we can talk about loose $\Zp$-lifts. Clearly,
      \begin{itemize}
        \item If a torsion $\Zp$-representation $\bar \rho$ admits loose (crystalline) $\Zp$-lifts, then $\bar\rho \otimes_{\Zp} k_E$ also admits loose (crystalline) $\O_E$-lifts.
        \item If a torsion $\O_E$-representation $\bar \rho$ admits loose (crystalline) $\O_E$-lifts, then if we regard $\bar \rho$ as a torsion $\Zp$-representation, it automatically admits loose (crystalline) $\Zp$-lifts.
      \end{itemize}
      So in general, when we say a torsion representation $\bar\rho$ admits loose (crystalline) lifts, it does not matter if we are regarding $\bar\rho$ as a $\Zp$-module or an $\O_E$-module.
\end{enumerate}
\end{remark}

The main theorem in this section is that, for $\bar \rho$ a torsion Galois representation, crystalline loose lifts always exist.
The proof uses an induction on $n$ where $p^n$ kills $\bar \rho$. The $n=1$ case will be (easily) deduced from the following Theorem \ref{thm-potential-lift}.

\subsection{Potentially strict lift for $p$-torsion representation}\label{subsect-pot-lift}
\begin{theorem}\label{thm-potential-lift}
Suppose $\bar \rho : G_K \to \GL_d (k_E)$ is a Galois representation where $E$ is big enough. Then there exists a finite unramified extension ${K'}$ of $K$, so that $\bar \rho |_{G_{{K'}}}$ admits a strict crystalline $\O_E$-lift.
\end{theorem}

\begin{proof}
\textbf{Step 1.} We will first show that there exists a lift, and make it to be crystalline in the final step.
By Definition \ref{convention: big enough E}, we can suppose that $\bar \rho $ is a successive extension of a sequence of irreducible representations $\Vbar _i $. And for each $i$, we let $V_i$ denote a fixed strict crystalline $\O_E$-lift (with non-positive Hodge-Tate weights) of $\Vbar _i$.
It will suffice to prove the following:
\begin{itemize}[leftmargin=*]
  \item (Statement A): There exists a finite unramified extension ${K'}/K$ such that $\bar \rho|_{G_{{K'}}}$ can be strictly lift to a representation $L'$,  which is a successive extension of $\chi_i V_i$, where $\chi_i$ are characters satisfying $\chi_i \equiv 1 \mod \varpi$. Here $\varpi = \varpi_E$.
\end{itemize}

\textbf{Step 2.} By an obvious induction argument, (Statement A) is reduced to the following:
\begin{itemize}[leftmargin=*]
  \item (Statement B): Suppose that $\bar\rho$ sits in a short exact sequence of $G_K$-representations: $0\to\bar U \to \bar \rho \to \bar V \to 0$, where $U, V$ are $\O_E$-finite free $G_K$-representations. Then there exists a finite unramified extension $K'/K$, and a $G_{K'}$-representation $L'$ over $\O_E$ sitting in the short exact sequence  $0 \to U|_{G_{K'}} \to L' \to \chi V|_{G_{K'}} \to 0$ such that $\bar L=\bar\rho|_{G_{K'}}$.
\end{itemize}
Note that in order to use (Statement B) to prove (Statement A), we could have required either $\bar U$ or $\bar V$ to be irreducible; but actually we do not need this assumption.

As an extension, $\bar \rho$ corresponds to an element $\bar c$ in $H^1 (G_K, \Hom_{k_E} (\Vbar, \Ubar))$, and $\rho$ (if it exists) corresponds to  a element in $H^1 (G_{K'}, \Hom _{\cO_E} (\chi V , U ) )$.
First note the exact sequence
$$0 \to {\rm Hom}_{\cO_E} (\chi V, U) \overset {\varpi}{\to} {\rm Hom}_{\cO_E} (\chi V, U) \to {\rm Hom}_{k_E} (\Vbar, \Ubar) \to 0 .$$
induces the  exact sequence
$$ H^1 (G_K, \Hom _{\cO _E} (\chi V , U)) \to H^1 (G_K, \Hom_{k_E} (\Vbar , \Ubar ))\overset {\delta} {\to} H^2 (G_K, \Hom _{\cO_E}(\chi V , U))[\varpi]\to 0 $$
In general, the strategy is to select $\chi$ to make $\delta (\bar c) = 0 $. But we do not know if such $\chi$ always exists. After restricting everything to $G_{K'}$, we get the following commutative diagram:
$$\xymatrix@C=8pt {H^1 (G_K, \Hom _{\cO _E} (\chi V , U)) \ar[d]^{\res}\ar[r] &  H^1 (G_K, \Hom_{k_E} (\Vbar , \Ubar ))\ar[d]^{\res} \ar[r]^-{\delta} &  H^2 (G_K, \Hom _{\cO_E}(\chi V , U))[\varpi]\ar[d]^{\res}\ar[r] & 0 \\
H^1 (G_{K'}, \Hom _{\cO _E} (\chi V , U)) \ar[r] &  H^1 (G_{K'}, \Hom_{k_E} (\Vbar , \Ubar ))\ar[r]^-{\delta} &  H^2 (G_{K'}, \Hom _{\cO_E}(\chi V , U))[\varpi]\ar[r] & 0}  $$
So it suffices to choose a finite unramified extension $K'/K$, and a character $\chi$ so that $(\res \circ \delta) (\bar c) = 0$.
In fact, we will show that $\res: H^2 (G_K, \Hom _{\cO_E}(\chi V , U))[\varpi] \to  H^2 (G_{K'}, \Hom _{\cO_E}(\chi V , U))[\varpi]$ can be made into the zero map.

Now let $M:  = \Hom _{\cO_E} (\chi V , U)$ and $M^\vee : = \Hom_{\cO_E}(U , \chi V)  (1)$.
By Tate duality in Lemma \ref{lem:res-cor}, we are reduced to show (namely, the following implies Statement B):
\begin{itemize}[leftmargin=*]
  \item (Statement C): We can choose a finite unramified extension $K'/K$, and a character $\chi: G_K \to \cO^\times_E$ such that the corestriction map $\cor: (M ^\vee \otimes E / \cO_E)^{G_{K'}}/ \varpi \to  (M ^\vee \otimes E / \cO_E)^{G_K}/ \varpi$ is the zero map.
\end{itemize}

\textbf{Step 3} To prove (Statement C), we first let $\chi=\varepsilon_p^{-a}$ where $a>>0$ and $p-1\mid a$. We can choose $a$ large enough so that $ (M^\vee) ^{I_K} = \{0\}$ (just by making all Hodge-Tate weights of $M^\vee$ negative). This implies that $(M^\vee \otimes E/ \cO_E) ^{I_K}$ is a finite set, and so we can choose $K''/K$ finite and unramified such that $(M^\vee\otimes E/ \cO_E) ^{G_{K''}} =  (M^\vee \otimes E/ \cO_E) ^{I_K}$.

Let $K'/ K''$ be the unramified extension  with $[K':K'']= p$, and we verify Statement C in this situation.
Note that the corestriction map
$$\Cor : (M ^\vee \otimes E / \cO_E)^{G_{K'}}/ \varpi \to  (M ^\vee \otimes E / \cO_E)^{G_K}/ \varpi$$
is defined by the trace map.
Let $x \in (M^\vee \otimes E/ \cO_E) ^{G_{K'}}$, then  $\Cor (x) = \sum_{g \in \Gal (K'/K)} g(x)$.
Since we have $$ (M ^\vee \otimes E / \cO_E)^{G_{K'}} = (M ^\vee \otimes E / \cO_E)^{G_{K''}}  =(M^\vee \otimes E/ \cO_E) ^{I_K},$$
so  $g(x)= x$ for any $g \in \Gal (K'/K'')$. Hence,  $\Cor (x) = p \sum_{g \in \Gal (K''/K)} g(x)$. Since $x$ is also in $(M ^\vee \otimes E / \cO_E)^{G_{K''}}$, we see that $\sum_{g \in \Gal (K''/K)} g(x)$ is in
$ (M ^\vee \otimes E / \cO_E)^{G_K}$, and so  $\Cor (x)=0$. This completes the proof of Statement C.

\textbf{Step 4}
To make $L'$ to be crystalline, it suffices to choose $\chi=\varepsilon_p^{-a}$ where $a>>0$ and $p-1\mid a$. When $a$ is sufficiently large, e.g., when $-a+1$ is smaller than all Hodge-Tate weights of $U$, then $L'$ is crystalline by \cite[Prop 1.24, Prop 1.26]{Nek}.
\end{proof}

\begin{corollary} \label{rem min HT wts}
With notations as in Theorem \ref{thm-potential-lift}, then the strict crystalline lift of $\bar \rho|_{G_{K'}}$ can be constructed such that its Hodge-Tate weights are in the range $[-(p^{fd}-2), 0]$.
\end{corollary}
\begin{proof}
We prove by induction on the number of Jordan-H\"{o}lder factors of $\bar\rho$. When $\bar\rho$ is irreducible, this is Lemma \ref{rmk: big E}. When $\bar\rho$ sits in a short exact sequence as in (Statement B) above, then via induction hypothesis, the Hodge-Tate weights of $U$ (resp. $V$) are in the range $[-(p^{fd_1}-2), 0]$ (resp. $[-(p^{fd_2}-2), 0]$), where $d_1, d_2$ are the $k_E$-dimensions of $\bar U, \bar V$, and $d_1+d_2=d$. By the argument in Step 4 of the proof above, we can choose $a$ to be $p^{fd_1}+p-2$ (which is divisible by $p-1$), then it is easy to see that the minimal Hodge-Tate weight of $L'$ is $\ge -(p^{fd_2}-2) - (p^{fd_1}+p-2) \ge -(p^{fd}-2)$.
\end{proof}

\begin{lemma}\label{lem:res-cor}
Let $A$ be a finite $\O_E$-module killed by $p$-power with a continuous $\O_E$-linear $G_K$-action and $L/K$ a finite extension.  Write $A^\vee : = \Hom_{\cO_E} (A , E/ \cO_E) (1)$ the Tate dual of $A$.
Then the corestriction map $$\Cor: H^0 (G_L,  A^\vee) \to H^0 ( G_K, A^\vee)$$
 is the Tate dual of the map
 $$\res: H^2 (G_K, A) \to H^2 (G_L , A).$$
\end{lemma}
\begin{proof}
This is well-known for experts but we include the proof here for completeness.
Suppose that $A$ is killed by $p ^m$. The Tate duality is induced by cup product $\cup: H ^ 2 (G _K , A ) \times H ^0 (G_K , A^\vee) \to H^2 (G_K, \mu _{p^m} (\Qpbar)) \simeq \Z/ p^m \Z$. Indeed we have the following commutative diagram:
$$\xymatrix{ H^2 (G_K, A ) \ar[d]^{\res} & \times & H^0 (G_K, A^\vee )  \ar[r] ^-\cup & H^2 (G_K , \mu_{p ^m } (\Qpbar)) \\ H^2 (G_L , A ) &   \times & H^0 (G_L , A^\vee )\ar[u]^{\Cor}  \ar[r] ^-\cup & H^2 (G_L , \mu_{p ^m } (\Qpbar)) \ar[u]^{\Cor}  }  $$
That is, for any $x \in H^2 (G_K, A)$ and any $y \in H^0 (G_L , A^\vee)$, we have
\begin{equation}\label{eqn:cup}
\Cor (\res(x) \cup y) = x \cup \Cor (y).
\end{equation}
 This is proved in, for example, \cite[Prop 9 (iv)]{AW}. Now suppose there is a commutative diagram:
\begin{eqnarray}\label{diag:cor}
\begin{split}
\xymatrix{  H^2 (G_K , \mu_{p ^m } (\Qpbar)) \ar[r] ^-{\sim}  &  \Z/ p^m \Z \ar@{=}[d] \\
H^2 (G_L  , \mu_{p ^m } (\Qpbar)) \ar[u] ^{\Cor}  \ar[r] ^-\sim &  \Z/p^m \Z }
\end{split}
\end{eqnarray}
Note that  we can identify $H^0 (G_K, A^\vee) $ with $\Hom_{\Z/ p^m \Z} ( H^2 (G_K, A ), \Z/ p^m \Z)$ via duality induced by cup product. More precisely, any $y \in H^0 (G_K, A^\vee)$ is identified with the map $f_y: H^2 (G_K, A )\to \Z/ p^m \Z$; $x \mapsto x \cup y  $.
We  also  identify $H^0 (G_L , A^\vee) $ with $\Hom_{\Z/ p^m \Z} ( H^2 (G_L, A ), \Z/ p^m \Z)$ in a similar way. Now we see that the map $\res : H^2 (G_K, A) \to H^2 (G_L , A)$ induces a dual map $\Hom_{\Z/ p^m \Z} ( H^2 (G_L, A ), \Z/ p^m \Z) \to \Hom_{\Z/ p^m \Z} ( H^2 (G_K, A ), \Z/ p^m \Z)$ by  $f \mapsto f\circ \res.$ Then for $x \in H^2 (G_K, A)$, we have $(f_y \circ \res)  (x)= \res (x) \cup y  $, which is just $x \cup \Cor (y)$ by \eqref{eqn:cup} and diagram \eqref{diag:cor}. This means the dual map of $\res$ is $f_y \mapsto f_{\Cor (y)}$. This proves the lemma.
%For any $y  \in H^2 (G_L, A^\vee)$, we see that $f_y \mapsto f_y \circ \res$.

Now it suffices to prove that diagram \eqref{diag:cor} is commutative. First recall that from the main theorem of local class field theory, we have the following commutative diagram
 $$\xymatrix{H^2 (G_K , \Qpbar ^\times)\ar[d]^{\inv_K}  \ar[r]^-{\res} &  H^2 (G_L , \Qpbar ^\times)\ar[d]^{\inv_L}\\ \Q/\Z \ar[r]^{[L:K]} &  \Q/\Z }$$
and (here we use Hilbert 90)
 $$\xymatrix{ 0 \ar[r] & H^2 (G_K , \mu_{p^m }(\Qpbar))\ar[d] ^{\sim} \ar[r] & H^2 (G_K , \Qpbar ^\times)\ar[d]^{\inv_K}  \ar[r]^-{p^m } &  H^2 (G_K , \Qpbar ^\times)\ar[d]^{\inv_K} \ar[r]& 0
 \\0 \ar[r] & \Z/ p^m \Z \ar[r]^-\sim & \Q/\Z \ar[r]^{p^m} &  \Q/\Z \ar[r] & 0  }$$
So in particular, the isomorphism $ H^2 (G_K , \mu_{p^m }(\Qpbar)) \simeq \Z/ p^m \Z$ is induced by $\inv_K$. Since $\Cor \circ \res =[L:K] $, we have the following diagram
$$\xymatrix{H^2 (G_L , \Qpbar ^\times)\ar[d]^{\inv_L}  \ar[r]^-{\Cor} &  H^2 (G_K , \Qpbar ^\times)\ar[d]^{\inv_K}\\ \Q/\Z \ar[r]^{1} &  \Q/\Z }$$
Combing this diagram and previous diagram, we have proved diagram \eqref{diag:cor}.
\end{proof}

\subsection{Loose crystalline lifts}
\begin{thm}\label{thm-cryslift}
Suppose $\bar \rho : G_K \to \GL_d (\Fp)$ is a Galois representation. Suppose $E$ is big enough for $\bar \rho\otimes_{\Fp}k_E$.
Then there exists loose crystalline lifts for $\bar\rho$, which can be made to be finite free over $\O_E$.
\end{thm}
\begin{proof}
It suffices to show that $\bar \rho\otimes_{\Fp}k_E$ admits loose crystalline $\O_E$-lifts.
By Theorem \ref{thm-potential-lift}, there exists a ${K'}$ so that $\bar \rho\otimes_{\Fp}k_E|_{G_{{K'}}}$ admits a strict crystalline $\O_E$-lift $\rho'$.  Let $L =\Ind^{G_K}_{G_{{K'}}} \rho'$, then $ L / \varpi_E L = \Ind^{G_K}_{G_{{K'}} } \circ \Res^{G_K}_{G_{{K'}} } (\bar \rho\otimes_{\Fp}k_E)$,  which maps surjectively onto $\bar \rho\otimes_{\Fp}k_E$ (see Lemma \ref{lemma ind res}).
\end{proof}

\begin{thm}\label{thm p^n tor loose crys lift}
For $n \in \mathbb Z^{+}$, suppose $T_n$ is a $p^n$-torsion representation of $G_K$ ($T_n$ is not necessarily free over $\mathbb Z/p^n\mathbb Z$). Then there exists loose crystalline lifts for $T_n$.
If we let $T_1:=T_n/pT_n$ and suppose $E$ is big enough for $T_1\otimes_{\Fp}k_E$, then we can always make the loose crystalline lift to be finite free over $\O_E$.
\end{thm}

\begin{convention} \label{conv rep mod}
In the proof of Theorem \ref{thm p^n tor loose crys lift} and \S \ref{subsec exist h}, we will use a lot of representation as well as their associated modules (\'etale $\varphi$-modules or Kisin modules). We will adopt the following convention on notations of these representations and their associated modules (only in Theorem \ref{thm p^n tor loose crys lift} and \S \ref{subsec exist h}).
\begin{enumerate}
  \item We use a letter with $n$ on subscript to mean a torsion object killed by certain $p$-power. For example, we can write $A_n$ as a torsion representation, and $\mathcal A_n$ as its associated \'etale $\varphi$-module. $A_n$ is not necessarily killed by $p^n$ (although sometimes it is); the subscript $n$ could also mean the induction step.
      Note that there does not necessarily exist a finite free object having $A_n$ as its ``reduction".
  \item We use a letter with a tilde to mean a finite free object. For example, we can write $\widetilde B_n$ as a finite free crystalline $\Zp$-representation, $\wt{\mathcal B}_n$ as its associated \'etale $\varphi$-module, and $\wt{\mathfrak B}_n$ as its associated Kisin module. Note here that the subscript $n$ usually indicate the induction step.
  \item We use a letter with breve accent to mean an object having both free part and torsion part (actually, we only use this convention once). For example, we can write $\breve W_n$ as a representation with both $\Zp$-free part and $p$-power torsion part, and $\breve{\mathcal W_n}$ as its associated \'etale $\varphi$-module.
\end{enumerate}
Note that we only adopt this notation system with representations and modules, although it is compatible with usual notations on rings. For example, throughout the paper, we use notations like $\gs_n, \mathcal{O_{\mathcal E}}_n$ to mean reduction modulo $p^n$ of $\gs, \mathcal{O_{\mathcal E}}$.
\end{convention}

\begin{proof}[Proof of Theorem \ref{thm p^n tor loose crys lift}]
We prove by induction on $n$.
When $n=1$, this is Theorem \ref{thm-cryslift}. Suppose the statement is true for $n-1$. Denote $T_{n-1}:=T_n/p^{n-1}T_n$.
By induction hypothesis, we can suppose $f_{n-1} : \widetilde L_{n-1} \onto T_{n-1}$ is a loose crystalline lift such that $\widetilde L_{n-1}$ is finite free over $\O_E$.

Let $\breve W_{n}$ be the cartesian product of $\widetilde L_{n-1} \onto T_{n-1}$ and $T_n \onto T_{n-1}$. We have the following diagram of short exact sequences (of $\Zp[G_K]$-modules).
\begin{equation}\label{Diag-0}
\begin{split}
\xymatrix{ 0 \ar[r] & p^{n-1} T_{n}\ar[d] \ar[r] & \breve W_{n} \ar@{->>}[d]\ar[r] & \widetilde L_{n -1}\ar@{->>}[d] \ar[r] & 0 \\
0 \ar[r] & p^{n-1} T_{n} \ar[r] & T_{n} \ar[r] & T_{n-1} \ar[r] & 0}
\end{split}
\end{equation}
It is obvious that we have $p \widetilde L_{n-1} =p\breve W_{n}$, which has a section to $\breve W_{n}$. And we have the exact sequence
\begin{equation}\label{Eq-exactsq-W-n}
 0 \to p \widetilde L_{n-1}\to \breve W_{n} \to Z_{n} \to 0.
\end{equation}
Here $Z_{n}: = \breve W_{n}/ p \breve W_{n}$, and it sits in the following exact sequence of $G_K$-representations over $\Fp$:
 \begin{equation}\label{Eq-exactsq-Z-n}
 0 \to p ^{n-1} T_n \to Z_{n} \to \widetilde L_{n-1}/ p \widetilde L_{n-1} \to 0.
\end{equation}

Let us point out that at this point, there is a relatively shorter way to create a loose crystalline lift for $T_n$, see Remark \ref{remark shorter way}. However, we need to create a special kind of loose crystalline lift, whose construction will be critically used later (see the proof of Proposition \ref{prop-key}).

Since $T_1 \onto p ^{n -1} T_{n}$ (when $T_n$ is finite free over $\mathbb Z/p^n\mathbb Z$, this surjection is bijective), so $E$ is also big enough for $p ^{n -1} T_{n}\otimes_{\Fp}k_E$.
By Theorem \ref{thm-potential-lift}, there exists $K''/K$ finite unramified such that there exists a strict crystalline $\O_E$-lift
$\widetilde N_{n}''  \onto p ^{n -1} T_{n}\otimes_{\Fp}k_E|_{G_{K''}} $.
Also, $\widetilde L_{n-1}$ is a strict crystalline $\O_E$-lift of $\widetilde L_{n-1}/\varpi_E\widetilde L_{n-1}$. By (Statement B) in the proof of Theorem \ref{thm-potential-lift}, there exists $K'/K''$ finite unramified and $s>>0$, such that we have the following diagram of short exact sequences of $G_{K'}$-representations:
\begin{equation}\label{Diag-2 first apperance}
\begin{split}
\xymatrix{0 \ar[r] &  \widetilde N_{n}' (:=\widetilde N_{n}''|_{G_{K'}})   \ar@{->>}[d] \ar[r] &   \widetilde Z_{n}'  \ar@{->>}[d]\ar[r] & \widetilde L_{n-1}(-s)  \ar@{->>}[d]\ar[r]  & 0 \\
0 \ar[r] & p ^{n -1} T_{n}\otimes_{\Fp}k_E \ar@{->>}[d]\ar[r] & Z_{n}\otimes_{\Fp}k_E    \ar@{->>}[d]\ar[r] & \widetilde L_{n-1}/ p\widetilde L_{n-1}\otimes_{\Fp}k_E   \ar@{->>}[d]\ar[r] & 0\\
0 \ar[r] & p ^{n -1} T_{n}\ar[r] & Z_{n}   \ar[r] & \widetilde L_{n-1}/ p\widetilde L_{n-1}  \ar[r] & 0  ,}
\end{split}
\end{equation}
where the first row are strict crystalline $\O_E$-lifts of the second row, and the maps from the second row to the third row are compatible projections to chosen $\Fp$-direct summands. Note that here for $s$, by Corollary \ref{rem min HT wts}, we only need to have $(p-1)|s$ and $s \ge p^{fd}$ where $d=\dim_{\Fp} T_1 \ge \dim_{\Fp} p^{n-1}T_n$, \emph{i.e.}, we can simply let
\begin{equation} \label{eq s}
  s:= p^{fd}+p-2.
\end{equation}

Define $ \widetilde L_{n}': =  \widetilde Z_{n}' \times_{Z_n} \breve W_{n} $, where $\widetilde Z_{n}' \to {Z_n}$ comes from diagram \eqref{Diag-2 first apperance}.
The $G_{K'}$-representation $ \widetilde L_{n}'$ sits in the following diagram of short exact sequences:
\begin{equation}\label{Diag-1-NEW}
\begin{split}
\xymatrix{ 0 \ar[r] &  p \widetilde L_{n-1}|_{G_{K'}}\ar[d] \ar[r] & \widetilde L_{n}' \ar@{->>}[d]\ar[r] & \widetilde Z_{n}'\ar@{->>}[d] \ar[r] & 0 \\   0 \ar[r] & p \widetilde L_{n-1}|_{G_{K'}} \ar[r] & \breve W_{n}|_{G_{K'}} \ar[r] & Z_{n}|_{G_{K'}} \ar[r] & 0}
\end{split}
\end{equation}
Clearly $\widetilde L_{n}'$ has to be $\mathcal O_E$-finite free, because it is extension by two $\mathcal O_E$-finite free modules. Note that $\widetilde L_{n}'[\frac 1 p]= \widetilde Z_{n}'[\frac 1 p] \times_{Z_n[\frac 1 p]} \breve W_{n}[\frac 1 p]
 \simeq \widetilde Z_{n}'[\frac 1 p] \oplus \widetilde L_{n-1} [\frac 1 p]$, so $\widetilde L_{n}'$ is crystalline. Clearly,
$\widetilde L_{n}'$ maps surjectively onto $\breve W_n|_{G_{K'}}$, and so onto $T_n |_{G_{K'}}$.

Let $\widetilde L_{n}:= \Ind_{G_{K'}}^{G_K} \widetilde L_{n}' $; it is a crystalline $G_K$-representation, and it maps surjectively to $T_n$.
\end{proof}

\begin{remark}\label{remark shorter way}
 As we point out earlier, after equation \eqref{Eq-exactsq-Z-n}, we could give a shorter proof of the above theorem. Indeed, we could directly use a loose $G_K$-crystalline lift $\widetilde Z_{n}^{(2)}$ of $Z_n$, and take $\widetilde L_{n}^{(2)}$ as the pull back in the following.
\begin{equation}\label{Diag-1}
\begin{split}
\xymatrix{ 0 \ar[r] &  p \widetilde L_{n-1}\ar[d] \ar[r] & \widetilde L_{n}^{(2)} \ar@{->>}[d]\ar[r] & \widetilde Z_{n}^{(2)}\otimes_{\Zp}\O_E \ar@{->>}[d] \ar[r] & 0 \\
0 \ar[r] & p \widetilde L_{n-1} \ar[r] & \breve W_{n} \ar[r] & Z_{n} \ar[r] & 0}
\end{split}
\end{equation}
Similarly, $\widetilde L_{n}^{(2)}$ is $\mathcal O_E$-finite free and crystalline, and it is a loose crystalline lift of $T_n$.
\end{remark}

\begin{remark} \label{rem linear growth}
\begin{enumerate}
  \item For any $n \ge 1$, let $-h_n$ be the minimal Hodge-Tate weight of $\widetilde L_{n}$ .
Then from the construction in diagrams \eqref{Diag-2 first apperance} and \eqref{Diag-1-NEW}, it is easy to see that $h_n=h_{n-1}+s$. So we have $h_n = h_1 +(n-1)s, \forall n \ge 1$, \emph{i.e.}, $h_n$ grows \emph{linearly}!  We will also see later (e.g., in the proof of Lemma \ref{lemma bound h}), that the ``growth rate" $s$ plays a role in overconvergence.

  \item  By Corollary \ref{rem min HT wts} and Equation \eqref{eq s}, we have
  $h_n \le n(p^{fd}+p-2)-p.$

  \item Suppose we have a finite free $\Zp$-representation $T$, and let $T_n=T/p^nT$. As we see,
  the ``crystalline weight" $h_n$ of $T_n$ grows linearly. It is intriguing to ask if $h_n$ can grow even more slowly. By the main result in \cite{Gao17}, the linear growth is indeed the best we can hope for; in fact, if $h_n$ grows in a certain log-growth as in \emph{loc. cit.}, then $T$ will be forced to be crystalline (or semi-stable, in the more general setting of \emph{loc. cit.}).
\end{enumerate}
\end{remark}

\section{Maximal Kisin models} \label{sec max kisin}
In this section, we study liftable Kisin models in torsion \'etale $\varphi$-modules, and show that they admit a (unique) maximal object. We study the finite free $\gs_1$-pieces of these maximal models. We also show that these maximal models are ``invariant" under finite unramified base change.

\subsection{Maximal Kisin models and devissage}
For $n \in \mathbb Z^{+}$, suppose $T_n$ is a $p^n$-torsion representation of $G_K$ ($T_n$ is not necessarily free over $\mathbb Z/p^n\mathbb Z$). Let $M_n$ be the corresponding \'etale $\varphi$-module. Recall that in the torsion case, a Kisin module $\M \subset M_n $ is called a \emph{Kisin model} if $\M[\frac 1 u] = M_n$.

\begin{defn}
  A Kisin model $\M$ is called \emph{loosely liftable} (in short, liftable), if there exists a $G_{K}$-stable $\Z_p$-lattice $L$ inside a crystalline representation and surjective map $f: L \sto T_n$ such that the corresponding map of \'etale $\varphi$-modules $f : \mathcal L \sto  M_n$  satisfies $f(\L)=\M$, where $\mathcal L$ (resp. $\L$) is the \'etale $\varphi$-module (resp. Kisin module) for $L$ (see the discussion before Example \ref{ex-cyclotomic}).
In this case, we say that $\M$ can be \emph{realized} by the surjection $f : L \sto T_n$.
\end{defn}

\begin{lemma} \label{lem-liftexist} \begin{enumerate}\item For any $M_n$, a liftable model $\M \subset M_n$ exists.
\item The set of liftable models inside $M_n$ has a unique maximal object (which we will denote as $\M_{(n)}$).
\end{enumerate}
\end{lemma}
\begin{proof}
 Item (1) is direct corollary of Theorem \ref{thm p^n tor loose crys lift} (see also the paragraph after Lemma \ref{lem model}).

For Item (2),
denote the set of all liftable models inside $M_n$ as $\textnormal{LF}_{\gs}^{\infty}(M_n)$. Indeed, this notation imitates the notations in \cite[Def. 3.2.1]{liu-car1}, where $\textnormal{F}_{\gs}^{\infty}(M_n)$ denotes the set of all Kisin models (not necessarily liftable).
Let us emphasize here that in \cite[Def. 3.2.1]{liu-car1}, $r$ is allowed to be $\infty$ (\emph{cf.} \cite[\S 2.1]{liu-car1}).
Clearly, $\textnormal{LF}_{\gs}^{\infty}(M_n)$ (which is a subset of $\textnormal{F}_{\gs}^{\infty}(M_n)$) is partially ordered.

Let$\M$ and $\M'$ be two liftable Kisin models of $M_n $ realized by $f : L \to T_n$ and $f' : L'\to T_n$ respectively. Then the lift $g: L \oplus L' \to T_n $ defined by $g (x, y)= f(x)+ f'(y)$ realizes $\M+\M'$. So the set $\textnormal{LF}_{\gs}^{\infty}(M_n)$ admits finite supremum.

For a chosen $\m \in \textnormal{LF}_{\gs}^{\infty}(M_n)$, any ascending chain in $\textnormal{LF}_{\gs}^{\infty}(M_n)$ of the form $\M \subset \M' \subset \M'' \cdots \subset M_n$ must stablize after finite steps, since by \cite[Cor. 3.2.6]{liu-car1}, it stablizes after finite steps as a chain in $\textnormal{F}_{\gs}^{\infty}(M_n)$. This implies that $\textnormal{LF}_{\gs}^{\infty}(M_n)$ has a unique maximal object.
\end{proof}

\subsubsection{Some notations}\label{subsubsec notations pieces}
In this subsubsection, we introduce some notations which will be used later.

Let $\M \in {'\sfi}$ be a torsion Kisin module such that $M: = \M[\frac 1 u] $ is a finite free $\gs_n [\frac 1 u]$-module (i.e., the torsion $G_\infty$-representation associated to $M$ is finite free over $\mathbb Z/p^n \mathbb Z$).
In general, $\M$ does not have to be $\gs_n$-free. By we can always use the following technique to d\'evissage $\M$ to finite $\gs_1$-free pieces.

For $i \le n$, let $M_i:=M/p^i M$, and let $q_i : M \to M _i$ be the natural projection. Then $q_i ( \M) \subset M_i$ is a Kisin model of $M_i$. Obviously, $q_i (\M)$
is the image of the natural map $\M/ p ^i\M \to M_i$. Following the discussion above \cite[Lem. 4.2.4]{liu2}, for each $0 \le i < j \le n$, we define
$$ \M ^{i, j}:= \t{Ker}  (p ^i \M \overset{p ^{j-i}}{\longrightarrow} p ^ j \M ).$$
By  the natural isomorphism $p ^{n-i}M/ p ^{n -j} M \simeq M_{i-j}$, we have $p ^{n-i} \M\simeq q_i  (\M)$ and  $p ^l \M^{i, j} \simeq \M ^{i+l , j}$.
Hence we have
$$p ^{n-1}\M = \M ^{n-1, n} \subset \M ^{n-2, n-1} \subset \cdots \subset \M^{0,1 }\subset M_1. $$
Indeed, it is not hard to see that $$\M ^{i, i +1} = p ^i \M \cap \Ker (p) = p ^i \M \cap p ^{n -1} M. $$
Note that $q_i (\M)$, $\Ker (q_i|_{\M})$ and $\M^{i, i+1}$ are all objects in $'\sfi$.

\subsubsection{Some more notations}\label{subsubsec notations free T}
In this subsubsection, we introduce some more notations which will be used later.

Let $T$ be a $\Zp$-finite free $G_K$-representation, and set $T_n : = T/p ^n T$.
Let $M$ be the finite free \'etale $\varphi$-module corresponding to $T$, and set $M_n:=M/p^n M$.

Denote the natural projection $q_{j ,i } : M_j \onto  M_i$ for $i < j$ induced by modulo $p^i$.
Recall that we use $\M_{(j)}$ to denote the maximal liftable Kisin model of $M_j$.
Set $\M _{(j , i ) } = q_{j ,i} (\M_{(j)})$.
It is easy to see that $\M_{(j , i)}$ is liftable. Indeed, if $f : L \to T_j$ realize $\M_{(j)}$, then $f' : L \to T_j \onto T_i$ realize $\M_{(j, i)}$.
So $\M_{(j, i)} \subset \M_{(i)}$ by construction.

For $i<j$, we denote $\iota_{i, j}: M_i \to M_j$ the injective map where for $x \in M_i$, we choose any lift $\hat x \in M_j$, and let $\iota_{i, j}(x) =p^{j-i} \hat x$. This is clearly well-defined, and we will use it to identify $M_i$ with $\iota_{i, j}(M_i)= M_j[p^i]=p^{j-i}M_j$ (recall that the notation $M_j[p^i]$ denotes the $p^i$-torsion elements). The maps $\iota_{i, j}$ are clearly transitive; namely, $\iota_{i, j}\circ \iota_{j, k}=\iota_{i, k}$. Also, the composite
$ M_i \overset{\iota_{i, j}}{\longrightarrow}  M_j \overset{q_{j, i}}{\longrightarrow}   M_i $ is precisely the map $ M_i \overset{\times p^{j-i}}{\longrightarrow} M_i$.

\begin{lemma}
Use notations in \ref{subsubsec notations free T}. In particular, for $i<j$, we identify $M_i$ with $M_j[p^i]$. Then $\m_{(j)}[p^i] = \m_{(i)}$ as Kisin models of $M_i$.
\end{lemma}
\begin{proof}
Suppose $f: L \onto T_j$ realizes $\m_{(j)}$. Let $g: L \onto T_j \onto p^iT_j$ be the composite map (where the second map is the $\times p^i$ map), and let $K: =\Ker g$. We have the following commutative diagram
\begin{equation}
\begin{split}
\xymatrix{ 0 \ar[r] & K\ar[d]^{f} \ar[r] & L \ar@{->>}[d]\ar[r]^{g} & p^iT_j\ar[d]^{=} \ar[r] & 0 \\   0 \ar[r] & T_j[p^i] \ar[r]& T_j \ar[r] & p^iT_j \ar[r] & 0 }
\end{split}
\end{equation}
where both rows are short exact sequences of $G_K$-representations.
The above diagram induces the following diagram of Kisin modules:
\begin{equation}\label{diag blah}
\begin{split}
\xymatrix{ 0 \ar[r] & \mathfrak K\ar@{->>}[d]^{f} \ar[r] & \mathfrak L \ar@{->>}[d]\ar[r]^{g} &
p^i\m_{(j)}\ar[d]^{=} \ar[r] & 0 \\
& f(\mathfrak K) \ar[r]& \m_{(j)} \ar[r] & p^i\m_{(j)} \ar[r] & 0 }
\end{split}
\end{equation}
Now, the top row of \eqref{diag blah} is short exact by \cite[Thm. 3.1.3(3), Lem. 3.1.4]{liu-car2} (or see the nice summary in \cite[Thm. 5.2]{GLS}). This implies that $f(\mathfrak K)= \m_{(j)}[p^i]$, and so $\m_{(j)}[p^i]$ is a liftable model. By maximality of $\m_{(i)}$, we have $\m_{(j)}[p^i] \subset \m_{(i)}$.

For the other direction, it is clear that $\m_{(i)}+\m_{(j)}$ is a liftable model in $M_j$. So we have $\m_{(i)}\subset \m_{(j)}$, and so $\m_{(i)}\subset \m_{(j)}[p^i]$.
\end{proof}

\begin{corollary} \label{cor identify}
With notations in \ref{subsubsec notations free T}, we have $ \m_{(j)}^{i-1, i} = \m_{(i, 1)}  $ for $j\ge i$.
\end{corollary}
\begin{proof}
By definition, we have $ \m_{(j)}^{i-1, i}  = p^{i-1}(\m_{(j)}[p^i])$. This is equal to $p^{i-1}\m_{(i)}$ (=$ \m_{(i, 1)}$) by the above lemma.
\end{proof}

\subsection{\'Etale descent of  Kisin modules and Kisin models}
Since the constructions in the previous section involve restriction and induction of representations over finite unramified extensions, in this subsection, we discuss how these operations affect the corresponding Kisin modules.
Since we are only dealing with unramified field extension, the situation here is not too difficult.

Let ${K'}$ be a finite unramified extension of $K$ with residue field ${k'}$,  $K'_0 = W({k'})[\frac 1 p]$ and  $\Gamma : = \gal ({K'}/K) \simeq \gal (K'_0/K_0) \simeq \gal ({k'}/k)$. Set $\gs ' = W(k ') \otimes_{W(k)} \gs$. Then $\Gamma$ acts on $\gs'$ and $(\gs') ^\Gamma = \gs$.

Set  $K'_\infty: =K'K_\infty$, and since $K'\cap K_\infty=K$, we have $G'_\infty : =\Gal(\overline K/K'_\infty)= G_{K'} \cap G_\infty$, and $G_\infty/G'_\infty \simeq \Gamma$. For each element $\gamma \in \Gamma$, we fix a lift in $G_\infty$, which we still denote as $\gamma$. In the remainder of this subsection, without further notice, we will use $\gamma$ to mean its lift in $G_\infty$. Since $G_\infty$ acts on $u=[\underline \pi]$ trivially, the $\Gamma$ action we mentioned in the previous paragraph is now the same as the induced action from $G_\infty$ on $\gs \subset W(R)$.

Since we will use induction and restriction of representations a lot, we include the following easy lemma.

\begin{lemma} \label{lemma ind res}  Let $G$ be a topological group, $H \subset G$ a closed normal subgroup of finite index. We use $\Ind$ and $\Res$ to denote the functors $\Ind_H^G$ and $\Res_H^G$ respectively.

Suppose $T$ is a representation of $H$ and $\rho_T : H \to \GL(T)$ denote the action $H$ on $T$. For any $\gamma \in G/H$, let $T^{\gamma}$ be the representation of $H$ acting on $T$ so that  $\rho_{T^\gamma} (h) = \rho_T({\gamma^{-1} h \gamma})$.

  \begin{enumerate}
    \item The functor $\Ind$ is both left and right adjoint of $\Res$.
    \item If $V$ is a $G$-representation, then $\Ind\circ\Res V$ naturally maps surjectively onto $V$.
    \item If $W$ is a $H$-representation, then $\Res\circ \Ind W =\oplus_{\gamma \in G/H}W^\gamma$.
    \item If $V$ is a $G$-representation, then $\Res\circ \Ind \circ \Res V   =\oplus_{\gamma \in G/H} \Res V$.

        \item Suppose we have $G=G_K, H=G_{K'}$. If $L$ is a crystalline $G_{K'}$-representation, then so is $L^\gamma$.
  \end{enumerate}
\end{lemma}
\begin{proof}
(1) is because $H \subset G$ is of finite index. (2) is easy corollary of (1). (3) is via Mackey decomposition. (4) is easy corollary of (3).
To prove (5), note that by \cite[Lem. 2.2.9]{Patrikis}, $\Ind_{G_{K'}}^{G_K} L$ is crystalline, and so $\Res _{G_{K'}}^{G_K} (\Ind_{G_{K'}}^{G_K} L) = \oplus_{\gamma \in \Gamma} L^\gamma$ is a crystalline $G_{K'}$-representation, hence $L^\gamma$ is also crystalline.
\end{proof}

\begin{defn}
Suppose $\M'$ is a Kisin module over $\gs'$. We say that $\M'$ admits a \emph{descent data} if $\Gamma$ acts on $\M'$ semi-linearly and the action commutes with $\varphi_{\M'}$.
\end{defn}

Obviously, if $\M$ is a Kisin module over $\gs$, then $\M' = \gs' \otimes_{\gs} \M$ naturally admits a descent data, and so we can define a functor $\M \rightsquigarrow \M'$.

\begin{prop}\label{prop-etaledescent}
\begin{enumerate}
\item The functor $\M \rightsquigarrow \M' : =  \gs' \otimes_{\gs} \M$ induces an equivalence between the category of torsion Kisin modules over $\gs$ and the category of torsion Kisin modules over $\gs'$ with descent data. The quasi-inverse of the functor is given by
$\M' \rightsquigarrow (\M') ^\Gamma $. We will say $\M'$ \emph{descends} to $\M$.

\item $\M$ is finite $\gs_n$-free if and only if $\M'$ is $\gs'_n$-finite free.
\item $T^*_\gs (\M)|_{G'_\infty} \simeq T ^* _{\gs'} (\M ')$.
\end{enumerate}
\end{prop}
\begin{proof} (1) is a standard consequence of \'etale descent (see \cite[\S 6.2 Example B]{neron}) as $\gs'$ is finite \'etale over $\gs$.
For statement (2), we only need to prove the ``if" direction. We thank an anonymous referee for the following concise proof. Since $\M' =  \gs' \otimes_{\gs} \M$, we have $\M'/(u, p)\M' \simeq k'\otimes_k \M/(u, p)\M$. Lifting any $k$-basis of $\M/(u, p)\M$ (suppose the $k$-dimension is $d$) to $\M$ gives a surjective map $f:(\gs_n)^{\oplus d} \onto \M$. Consider the map $1\otimes f: W(k')\otimes_{W(k)} (\gs_n)^{\oplus d} = (\gs'_n)^{\oplus d} \onto W(k')\otimes_{W(k)} \M=\M'$. $1\otimes f$ is an isomorphism modulo $(u, p)$ and hence an isomorphism itself. Since $W(k')$ is faithfully flat over $W(k)$, $f$ is an isomorphism.
Statement (3) is clear by the definition of $T^*_\gs$ via \eqref{Eq-define TM}.
\end{proof}

\begin{co}\label{cor-descent-free} Suppose that $\M'$ is a finite free Kisin module over $\gs'$ with descent data. Then $\M'$ descends to a finite free Kisin module $\M$ over $\gs$.
\end{co}

\begin{co}\label{lem-descent}
Suppose that $T$ is a torsion representation of $G_K$ and $M$ is the corresponding \'etale $\varphi$-module.
Then $\Gamma$-acts on  $M': = W(k') \otimes_{W(k)}M$ semilinearly.
Suppose $\M' \subset M'$ is a Kisin model such that $\gamma (\M') \subset \M', \ \forall \gamma \in \Gamma$.  Then $\M'$ descends to a Kisin model $\M \subset M$.
\end{co}

\begin{defn}
Let $\M '$ be a Kisin module over $\gs'$. For any $\gamma \in \Gamma$, set $\M'_{\gamma} : = \gs' \otimes_{\gamma , \gs'}\M'$.
\end{defn}

\begin{lemma}  \label{lemma twist Kisin module}
  $T^*_{\gs'} (\M'_\gamma) \simeq (T^*_{\gs'} (\M'))^\gamma$ as $G'_\infty$-representations.
\end{lemma}
\begin{proof}
From the definition of $T^*_{\gs'}$, we can define a map from $T^*_{\gs'} (\M'_\gamma)$ to $(T^*_{\gs'} (\M'))^\gamma$ so that $\sum_i x_i \otimes m_i\mapsto \sum \gamma ^{-1} (x_i) \otimes m_i$ with $x_i \in W(\FrR)$ and $m_i \in \M'$. One can easily check that the map is an isomorphism of $G'_\infty$-representations.
\end{proof}

\subsubsection{} \label{subsubsec notation descent}
Suppose $T_n$ is a $p^n$-torsion representation of $G_K$, $M_n$ the corresponding \'etale $\varphi$-module. Then it is easy to see that $M'_n  := W(k ') \otimes _{W(k)} M_n $ is the corresponding \'etale $\varphi$-module for $T_n|_{G_{K'}}$.

\begin{lemma}\label{new-lem-realization-descent}
With notations in \ref{subsubsec notation descent}. Suppose $f: L \onto T_n|_{G_{K'}}$ is a surjection of $G_{K'}$-representations where $L$ is crystalline. Let $\mathfrak L$ be the Kisin module corresponding to $L$, and denote $\m':=f(\mathfrak L)$. If $\m'$ is $\Gamma$-stable as in Corollary \ref{lem-descent}, namely, $\m'$ descends to a $\gs$-module $\m$, then $\m$ is a liftable Kisin model of $M_n$.
\end{lemma}
\begin{proof}
Since $f: L \onto T_n|_{G_{K'}}$, so we have $\Ind_{G_{K'}}^{G_K} L  \onto \Ind_{G_{K'}}^{G_K} \Res_{G_{K'}}^{G_K} T_n \onto T_n$, where the last surjection is via Lemma \ref{lemma ind res}(2). It suffices to show that the composite $h: \Ind_{G_{K'}}^{G_K} L \onto T_n$ realizes $\M$. Suppose $h$ realizes $\M^{(2)}$, then we have $h|_{G_{K'}}$ realizes $W(k')\otimes_{W(k)}\M^{(2)}$. By Lemma \ref{lemma ind res}(3), $h|_{G_{K'}}$ clearly factors through $L \onto  T_n|_{G_{K'}}$, which realizes $\M'$, and so $\M'=W(k')\otimes_{W(k)}\M^{(2)}$. Since $W(k')$ is faithfully flat over $W(k)$, we must have $\M=\M^{(2)}$.
\end{proof}

Our next lemma shows that the maximal object $\M_{(n)}$ is compatible with unramified base change.
Recall that we use $\M_{(n)}$ and $\M'_{(n)}$ to denote the maximal liftable Kisin models of $M_n$ and $M_n'$ respectively.

\begin{lemma}\label{lem-max-descent}
With notations in \ref{subsubsec notation descent}, we have
$\M'_{(n)} \simeq W(k') \otimes_{W(k)} \M_{(n)}$.
\end{lemma}
\begin{proof}Since $\M_{(n)}$ can be realized by a loose crystalline lift $L \onto T_n$,   after restricting to $G_{K'}$, we see that $W(k') \otimes_{W(k)} \M_{(n)}  $ is liftable.
So $W(k') \otimes_{W(k)} \M_{(n)} \subset \M'_{(n)}$.
Conversely, we claim that $\M'_{(n)}$ is stable under $\Gamma$-action. If so, then by Lemma \ref{lem-descent}, $\M'_{(n)}$ descent to an $\M$ such that $\M$ is a Kisin model of $M_n$.
Furthermore, $\M$ is liftable by Lemma \ref{new-lem-realization-descent},  and so $\M \subset \M_{(n)}$, concluding the proof.

Now it suffices to prove the claim. Suppose $f: L \to T_n|_{G_{K' }}$ is surjection of $G_{K'}$-representations which realizes $ \M'_{(n)}$, then it is not hard (c.f. Lemma \ref{lemma twist Kisin module}) to see that $f ^{\gamma} : L ^{\gamma} \to T _n ^{\gamma}$ realizes $ \gamma (\M'_{(n)})$. That is to say, $\gamma (\M'_{(n)})$ is  liftable, and so $\gamma (\M'_{(n)}) \subset \M'_{(n)}$ by maximality of $\M'_{(n)}$ (and hence indeed, $\gamma (\M'_{(n)}) =  \M'_{(n)}$). In other words, $\M'_{(n)}$ is $\Gamma$-stable.
\end{proof}

\section{Torsion theory} \label{sec torsion thy}

In this section, we use the results on loose crystalline lifts to study Kisin models in the \'etale $\varphi$-modules corresponding to $p^n$-torsion representation of $G_K$. We use many facts on torsion Kisin modules heavily. The reader may consult \cite[\S 2.3]{liu2} for general facts on torsion Kisin modules.

\subsection{Generators of torsion Kisin modules}\label{subsect-generator}

\newcommand{\bolde}{\boldsymbol{e}}

\newcommand{\boldfe}{\boldsymbol{\mathfrak e}}

In this subsection, we will freely use notations from \ref{subsubsec notations pieces}. Recall that for $\m$ a $p^n$-torsion Kisin module such that $M: = \M[\frac 1 u] $ is a finite free $\gs_n [\frac 1 u]$-module, we have defined the modules $\m^{i, j}$ for $i<j$.
For each $1\le i \le n$, we can choose elements $\{\fe_j^{(i)} \in \m\}_{j=1}^d$, such that $\{p^{i-1}\fe_j^{(i)}\}_{j=1}^d$ forms a $\ku$-basis of $\m^{i-1, i}$.
The following easy lemma will be used later.

\begin{lemma}\label{lem-generators}
\begin{enumerate}
  \item For $m = 1, \dots , n$,  the module $\m[p ^{m}]$ is generated (over $\gs$) by $\{\fe_j^{(i)},  1\le i \le m\}_{j=1}^d$. Note that when $m=n$, $\m[p ^{m}]=\m$.
\end{enumerate}
\end{lemma}
\begin{proof} We prove by induction on $m$. The case $m = 1$ is trivial since $\m[p]$ is precisely $\m^{0, 1}$.
Suppose the statement is valid for $m-1 $, and consider $\m[p ^{m}]$.
For any $x \in \m[p ^{m}]$, we have $p^{m -1} x  \in \M ^{m -1, m}$. Hence   there exists $y_j  \in \gs$ so that  $ p ^{m -1}  (x - \sum_j y_j \fe^{(m)}_j ) = 0  $, and so $x - \sum_j y_j \fe^{(m)}_j \in \m[p ^{m -1}]$. By induction hypothesis, $x - \sum_j y_j \fe^{(m)}_j$ can be written as a linear combination of $\fe^{(i)}_{j}$ for $i \leq  m -1$; this completes the induction.
\end{proof}

The following lemma is the technical key of this subsection. Its assumption (``existence of $h$") will be verified for certain Kisin modules in Proposition \ref{prop-key}.

\begin{lemma}\label{lem-devissage} Using notations from above.
Suppose that there exists an $h \in \mathbb Z^{>0}$ such that $u ^h \M ^{i-1, i} \subset \M ^{i, i+1}$ for $i = 1, \dots, n-1$. Then the following holds.
\begin{enumerate}
\item For each $i$, we have $$(  \e^{(i)}_1, \dots \e^{(i)}_{d}) = (\e^{(n )}_1 , \dots , \e^{(n )}_d) (\frac {p}{u ^h})^{n-i}Y_{i, n}$$
     with $Y_{i, n} \in \Md (\gs[\frac{p}{u^{2h}}]) $.

\item For each $i= 1 ,\dots,  n -1$, we have  $$ p (\e^{(i+1)}_1 , \dots , \e_d ^{(i+1)}) = (\e ^{(i)}_1 , \dots , \e^{(i )}_d) \Lambda_{i  } \left( I_d +  \frac{p}{u ^{2h}} Y'_{i+1, i} \right) , $$
where $\Lambda_{i}\in \Md(\gs) $ such that $u ^h \Lambda _i ^{-1} \in \Md (\gs)$, and $Y'_{i+1, i} \in \Md (\gs[\frac{p}{u^{2h}}])$ .
\end{enumerate}

\end{lemma}
\begin{proof}
Since $u ^h \M ^{i-1, i} \subset \M ^{i, i+1} \subset \M ^{i-1, i}$, we obviously have
$$ (p ^{i  }\e^{(i+1)} _1 , \dots , p^{i  } \e^{(i+1 )}_d )= (p ^{i-1  }\e^{(i )} _1 , \dots , p ^{i-1}\e^{(i  )}_d ) \bar\Lambda _i, \ i = 1 , \dots , n-1, $$ where $\bar\Lambda_i \in \Md(\gs_1)$ is a matrix such that $u ^h \bar\Lambda _i ^{-1} \in \Md (\gs_1)$.

Now for each $i$, we construct a lift $\Lambda_i$ of $\bar \Lambda_i$ as follows.
Since $\ku$ is PID, we can write $\bar \Lambda_i = XAY$ where $X, Y$ are invertible matrices, and $A$ is diagonal matrix with elements on the diagonal of the form $u ^{a_t}$ such that $a_t \leq h$ for $1 \le t \le d$. Let $\tilde X, \tilde Y \in \Md(\gs)$ be some fixed lifts of $X$, $Y$ respectively. Then $\Lambda_i:= \tilde X A \tilde Y  \in \Md(\gs)$ is a lift of $\Lambda_i$ and satisfies $u ^h  \Lambda_i ^{-1} \in \Md (\gs)$. And so we have
$$ (p ^{i  }\e^{(i+1)} _1 , \dots , p^{i  } \e^{(i+1 )}_d )= (p ^{i-1  }\e^{(i )} _1 , \dots , p ^{i-1}\e^{(i  )}_d )  \Lambda _i, \ i = 1 , \dots , n-1, $$
since $p^i \e^{(i)}_j=0, 1\le j \le d$.

Now we prove by induction on $i$ ($\ge 2$) that
\begin{eqnarray} \label{induction eqn}
  (\e^{(i -k)}_1, \dots , \e^{(i -k)}_d ) = ( \e^{(i)}_1, \dots , \e^{(i)}_d) (\frac {p}{u ^h})^k  Y_{i-k, i}
\end{eqnarray}
with $Y_{i-k, i} \in \Md (\gs[\frac{p}{u^{2h}}]) $ for $1 \le k \le i-1$.
When $i =2$, we have  $(p \e^{(2)}_1 , \dots , p \e^{(2)}_d) = (\e^{(1)}_1, \dots , \e^{(1 )}_d) \Lambda_{1}$. The statement is valid as $Y_{1, 2}:=u ^h \Lambda_1 ^{-1}  \in \Md (\gs)$.
Now suppose the statement is valid for $i \le m -1$ ($m \ge 3$), we consider the situation of $i = m$. Since
$$(p^{m-1}  \e^{(m )}_1 , \dots , p^{m-1}  \e^{(m  )}_d) = (p ^{m -2 }\e^{(m -1 )}_1, \dots , p ^{m -2}\e^{(m -1  )}_d) \Lambda_{m -1}, $$
we have $p (\e_j^{(m)} )- (\e^{(m-1 )}_j ) \Lambda _{m -1} \in \M[p ^{m -2}] .$
By Lemma \ref{lem-generators}, $\M[p ^{m -2}]$ is generated by $ \e^{(i)}_j $ for $1 \leq i \leq  m - 2  $ and $j = 1, \dots , d$.
So, we have
\begin{eqnarray*}
p (\e_1^{(m)}, \dots , \e_d ^{(m)} ) -   (\e_1^ {(m-1 )}, \dots, \e_d ^{(m -1)} ) \Lambda _{m -1} =   \sum_{j =0} ^{m-2 } (\e_1 ^{(j)}, \dots \e_d ^{(j)}) X_j, \text{ for some } X_j \in \Md(\gs) \\
  =  \sum_{j =1} ^{m-2 } (\e_1 ^{(m -1) }, \dots , \e^{(m -1) }_d) (\frac {p}{u ^h })^{m-1-j }  Y_{j, m-1} X_j, \text{ by induction hypothesis.}
\end{eqnarray*}
Hence we have
\begin{eqnarray*}
 (\e_1^{(m)}, \dots , \e_d ^{(m)} )p\Lambda_{m -1} ^{-1} =    (\e_1^ {(m-1 )}, \dots, \e_d ^{(m -1)} )   \left( I_d +
    \sum_{j =1 } ^{m -2  }  (\frac {p}{u ^h })^{ m -1-j }    Y_{j, m-1} X_j \Lambda^{-1}_{m -1} \right).
\end{eqnarray*}
Since $u ^h \Lambda ^{-1}_{m -1}\in \Md (\gs)$, we can write (using $\frac{p}{u^h}=\frac{p}{u^{2h}} u^h$)
\begin{equation}\label{Eq-transit}
   (\e_1^{(m)}, \dots , \e_d ^{(m)} )p\Lambda_{m -1} ^{-1} =    (\e_1^ {(m-1 )}, \dots, \e_d ^{(m -1)} )
   \left( I _d + \frac{p}{u ^{2h}}  Y'_{m, m-1} \right),
\end{equation}
with $ Y'_{m, m-1} \in \Md(\gs[\frac{p}{u^{2h}}])$.
Hence
\begin{eqnarray*}
 (\e_1^ {(m-1 )}, \dots, \e_d ^{(m -1)} )&=&  (\e_1^{(m)}, \dots , \e_d ^{(m)} ) p \Lambda_{m -1} ^{-1}
  \left ( I _d + \frac{p}{u ^{2h}}  Y'_{m, m-1} \right)^{-1}\\
   &= & (\e_1^{(m)}, \dots , \e_d ^{(m)} ) \frac{p}{u ^h }
   Y_{m-1, m}
\end{eqnarray*}
with $Y_{m-1, m} \in \Md(\gs[\frac{p}{u^{2h}}])$ (Note that when calculating $(I _d + \frac{p}{u ^{2h}}  Y'_{m, m-1} )^{-1}$, we can throw away terms with high $p$-powers, because $p^m$ kills $(\e_1^{(m)}, \dots , \e_d ^{(m)} )$).
Now for $k =2, \dots , m-1$, we iterate the above  to get
\begin{eqnarray*}
 (\e_1 ^{(m -k)}, \dots \e_d ^{(m - k)})
  =   (\e_1 ^{(m - 1) }, \dots , \e^{(m -1) }_d) (\frac {p}{u ^h })^{k-1 }  Y_{m-k, m-1},
\end{eqnarray*}
and so \eqref{induction eqn} is proved.

It is clear \eqref{induction eqn} implies Item (1). And Item (2) is already proved in Equation \eqref{Eq-transit}.
\end{proof}

\subsection{Existence of $h$} \label{subsec exist h}
The following proposition will play a key role to prove the later overconvergence result. As we mentioned earlier, in this subsection, we will use the notation system in Convention \ref{conv rep mod}.

\begin{prop}\label{prop-key}
Use notations in \ref{subsubsec notations free T}.
There exists a constant $h$ only depending on $p$, $f$, $e$ and $d$ such that $u ^h \M_{(n-1, 1)} \subset \M_{(n ,1)}$ for all $n\geq 1$. Consequently $u ^h \M_{(n)} ^{i -1, i} \subset  \M_{(n)} ^{i , i +1}$ for all $i= 1, \dots , n -1$ by Corollary \ref{cor identify}.
\end{prop}

\begin{proof}
The proof is quite involved, so we break the proof into several steps. We first present the main strategy of the proof, assuming two difficult lemmas which will be proved later.

We first fix an $n$, so we can apply Theorem \ref{thm p^n tor loose crys lift} to our $T_n$, and we will freely use notations there.
By Lemma \ref{lem-max-descent}, for any $n$, if we let $M_n':=W(k')\otimes_{W(k)} M_n$, then the maximal liftable Kisin model of $M_n'$ is $\M'_{(n)} \simeq W(k') \otimes_{W(k)} \M_{(n)}$. Thus, it is easy to see that $\M'_{(n, 1)}=W(k')\otimes_{W(k)} \M_{(n, 1)} $ for any $n$. So to prove our proposition, it suffices to show that
$$u ^h \M'_{(n-1, 1)} \subset \M'_{(n ,1)}.$$

We divide the following argument into two steps. In Step 1, we will construct another Kisin model (denoted as $\mathring{\M}'_{n, 1}$) of $M_1'$ such that $\mathring{\M}'_{n, 1} \subset \M'_{(n ,1)}$. Then in Step 2, we show that $u ^h \M'_{(n-1, 1)} \subset \mathring{\M}'_{n, 1}.$

\textbf{Step 1.} The loose crystalline lift $\widetilde L'_n \onto \breve W_n|_{G_{K'}} \onto  T_n|_{G_{K'}}$ realizes a liftable Kisin model $\mathring{\M}'_n$ of $M'_n$, and so $\mathring{\M}'_n \subset \M_{(n)}'$.
The following composite of $G_{K'}$-representations
\begin{equation} \label{easy 1}
\xymatrix{
\widetilde L'_n\ar@{->>}[r] & \breve W_n|_{G_{K'}} \ar@{->>}[r] &T_n|_{G_{K'}} \ar@{->>}[r] & T_1|_{G_{K'}}
.}
\end{equation}
realizes a Kisin model $\mathring{\M}'_{n, 1}$ in $M_1'$, and we certainly have
$$\mathring{\M}'_{n, 1} \subset \M'_{(n, 1)}.$$

By diagram \eqref{Diag-0}, The composite \eqref{easy 1} is the same as
\begin{equation*}
\xymatrix{
\widetilde L'_n \ar@{->>}[r]  & \breve W_n \ar@{->>}[r] & \widetilde L_{n -1}  (\ar@{->>}[r] & T_{n-1}) \ar@{->>}[r] & T_1,
}
\end{equation*}
which is the same as
\begin{equation*}
\xymatrix{
\widetilde L'_n \ar@{->>}[r]  &  \breve W_n \ar@{->>}[r]  &  \breve W_n/p\breve W_n=Z_n \ar@{->>}[r]  & \widetilde L_{n -1}/p\widetilde L_{n -1}  \ar@{->>}[r]  & T_1,
}
\end{equation*}
which, by \eqref{Diag-1-NEW}, is the same as
\begin{equation*}
\xymatrix{
\widetilde L'_n \ar@{->>}[r]  & \widetilde Z_n'  \ar@{->>}[r]  & Z_n \ar@{->>}[r]  & \widetilde L_{n -1}/p\widetilde L_{n -1}  \ar@{->>}[r]  & T_1,
}
\end{equation*}
which, by diagram \eqref{Diag-2 first apperance}, is the same as
\begin{equation}\label{Eq-sequence-m'}
 \xymatrix{\widetilde L'_n\ar@{->>}[r] &  \widetilde Z_n' \ar@{->>}[r] & \widetilde L_{n -1}(-s)  \ar@{->>}[r] &\widetilde L_{n -1}/p\widetilde L_{n -1}  \ar@{->>}[r] & T_1.}
\end{equation}
And so \eqref{Eq-sequence-m'} also realizes $\mathring{\M}'_{n, 1}$.

\textbf{Step 2.} The last surjection of \eqref{Eq-sequence-m'} is induced by the following composite of $G_{K}$-representations.
\begin{equation}\label{Eq-sequence-mn-1}
\xymatrix{\widetilde L_{n-1}\ar@{->>}[r] & T_{n-1} \ar@{->>}[r] & T_1}
\end{equation}
Since we may assume that $\wt L_{n-1}\twoheadrightarrow T_{n-1}$ realize $\M_{(n -1)}$, and the composite \eqref{Eq-sequence-mn-1} restricted to $G_{K'}$ realizes $\M'_{(n-1, 1)}$.

So in order to prove $u ^h \M'_{(n-1, 1)} \subset \M'_{(n ,1)}$, it suffices to show $u ^h \M'_{(n-1, 1)} \subset \mathring{\M}'_{n, 1}$.
And so it suffices to show that the cokernel of the following composite (which are maps of Kisin modules corresponding to \eqref{Eq-sequence-m'}) is killed by $u^h$
\begin{equation}\label{Eq-uh}
\xymatrix{\widetilde{\L}'_{n} \ar[r] &  \widetilde{\mathfrak{Z}}_n' \ar[r] & \widetilde \L'_{n -1}(s)     \ar[r] & \L'_{n-1}/p\L'_{n-1}}.
\end{equation}
By Lemma \ref{lem-surjection} below, the map $\widetilde{\L}'_{n} \to  \widetilde{\mathfrak{Z}}_n'$ is in fact surjective, so we only need to consider the cokernel of the following composite of maps:
\begin{equation}\label{Eq-uhuh}
\xymatrix{\widetilde{\mathfrak{Z}}_n' \ar[r] & \widetilde \L'_{n -1}(s)     \ar[r] & \L'_{n-1}/p\L'_{n-1}}.
\end{equation}
Denote the image of the composite \eqref{Eq-uhuh} as $\mathring \L'_{n-1}$, which is contained in $\L'_{n-1}(s)/p\L'_{n-1}(s)$.
So we can choose basis $e_1, \ldots e_m$ of $\L'_{n-1}$, such that $\mathring \L'_{n-1}$ has a $k'[\![u ]\!]$-basis formed by $u^{a_1}\bar e_1, \ldots u^{a_m}\bar e_m$, where $a_i=\frac{es}{p-1}+b_i$ with $b_i\ge 0$ (note that the rank of $\mathring \L'_{n-1}$ is also $m$, because of the surjective maps in \eqref{Eq-sequence-m'}).

Finally, we will show in Lemma \ref{lemma bound h} that $a_i$'s are bounded by a constant $h$, and this will conclude the proof.
\end{proof}

Now we will prove Lemma \ref{lem-surjection} and Lemma \ref{lemma bound h} to complete the proof of Proposition \ref{prop-key}.
Before we prove Lemma \ref{lem-surjection}, let us first recall the following subtle fact: suppose that we have an exact sequence of lattices inside semi-stable representations:
$$0 \to L^{(1)} \to L^{(2)}\to L ^{(3)} \to 0$$
Then the corresponding sequence of Kisin modules
$$0\to \L^{(1)} \to \L^{(2)} \to \L^{(3)} \to 0 $$ is only \emph{left} exact \cite[Lem. 2.19.]{liu-wd}, \emph{i.e.}, the map $\L^{(2)} \to \L^{(3)}$ is not necessarily surjective.

\begin{lemma}\label{lem-surjection}
The sequence of finite free Kisin modules corresponding to the first row of Diagram\eqref{Diag-1-NEW}
$$ 0 \to p \widetilde \L'_{n-1} \longrightarrow \widetilde \L'_{n}  \overset{\theta}\longrightarrow \widetilde{\mathfrak{Z}}'_{n} \longrightarrow 0 $$
is exact, \emph{i.e.}, $\theta$ is surjective.
\end{lemma}

\begin{convention}
In the proof of this lemma, \emph{all} the representations that we consider are $G_{K'}$-representations, and all Kisin modules are over $\gs'=W(k')\otimes_{W(k)} \gs$. To be completely rigorous, we will need to restrict many representations from $G_K$ to $G_{K'}$, and use prime notation over Kisin modules (\emph{i.e.}, notations like $\M'$). For notational simplicity, from now on (\emph{i.e.}, in the proof of Lemma \ref{lem-surjection} and Lemma \ref{lemma bound h}), we will drop these prime notations.
\end{convention}

\begin{proof}
\textbf{Step 0:} \emph{strategy.} We first sketch the strategy of the proof as follows.
We will not directly show that $\theta$ is surjective. Instead, we will construct a Kisin model $\breve \W_{n}$ inside $\breve \cW_{n}$, which sits in a short exact sequence
\begin{equation}\label{Eq-exactsq-gW-n}
0 \to p \widetilde \L_{n-1} \to\breve \W_{n} \to \gZ_{n} \to 0,
\end{equation}
corresponding to the exact sequence \eqref{Eq-exactsq-W-n} (restricted to $G_{K'}$),  where $\gZ_n$ is the Kisin model realized by $\wt Z_n\twoheadrightarrow Z_n$.

We then define the product $\widetilde \L^{(2)}_{n}:=  \breve \W_{n} \times_{\mathfrak Z_n} \widetilde{\mathfrak Z}_n$, which is a Kisin model of $\widetilde{\mathcal L}_n$, and naturally sits inside the short exact sequence:
$$ 0 \to p \widetilde \L_{n-1} \longrightarrow \widetilde \L^{(2)}_{n}  \overset{\theta}\longrightarrow \widetilde{\mathfrak{Z}}_{n} \longrightarrow 0. $$
This forces $\widetilde \L^{(2)}_{n}$ to be finite free, and so it has to be equal to $\widetilde \L_{n}$ (because the finite free Kisin model of $\mathcal{L}_n$ is unique), concluding the proof.

\textbf{Step 1:} \emph{a basis for $\mathfrak Z_n$}. As the strategy suggests, we first construct a basis for $\gZ_{n}$.

From diagram \eqref{Diag-2 first apperance}, we have the following
\begin{equation}\label{Diag-2}
\begin{split}
\xymatrix{0 \ar[r] &  \widetilde N_{n} \ar@{->>}[d] \ar[r] &   \widetilde Z_{n}  \ar@{->>}[d]\ar[r] & \widetilde L_{n-1}(-s) \ar@{->>}[d]\ar[r]  & 0 \\
0 \ar[r] & p ^{n -1} T_{n} \ar[r] & Z_{n}  \ar[r] & \widetilde L_{n-1}/ p\widetilde L_{n-1} \ar[r] & 0 .}
\end{split}
\end{equation}
Note that in this diagram, only the map $\widetilde L_{n-1}(-s) \onto \widetilde L_{n-1}/ p\widetilde L_{n-1}$ is a \emph{strict} crystalline $\Zp$-lift (\emph{i.e.}, it is a $\bmod p$ map); the other two projections $\widetilde N_{n}\onto p ^{n -1} T_{n}$ and  $\widetilde Z_{n} \onto Z_{n}$ are just loose crystalline lifts.

Recall that as in Step 2 of the proof of Proposition \ref{prop-key}, $\mathring \L_{n-1}$ is the torsion Kisin module induced by the surjection $\widetilde Z_{n} \onto \widetilde L_{n-1}(-s) \onto \widetilde L_{n-1}/ p\widetilde L_{n-1}$.
Let $\gZ _{n}$ be the torsion Kisin module induced by the map $\widetilde Z_n \onto Z_{n}$.
Let $\N_{n}$ be the kernel of the surjection $\gZ_{n} \onto  \mathring \L_{n-1}$.
The following short exact sequence
\begin{equation}\label{seq-1}
0 \to  \N_{n} \to \gZ_{n} \to  \mathring \L_{n-1} \to 0.
\end{equation}
corresponds to the bottom row of \eqref{Diag-2}, and so $\N_{n}$ is a Kisin model of $p^{n-1}M_n$, and hence finite free over $\ku$.
We already have the basis $u ^{a_1} \bar e_1, \dots , u^{a_m} \bar e_m$ for $\mathring \L_{n-1} $.
Pick any $\ku$-basis $\alpha_1, \dots , \alpha_t$ for $\N _{n}$, and then $\{\alpha_1, \dots , \alpha_t, u ^{a_1} \bar e_1, \dots , u^{a_m} \bar e_m\}$ form a basis of $\gZ_{n}$.

\textbf{Step 2:} \emph{a Kisin model inside $\breve \cW_{n}$.}
Now we construct a Kisin model $\breve \W_{n}$ inside $\breve \cW_{n}$.
Consider $e_1 , \dots , e_m$ the $\gs$-basis of $\widetilde \L_{n-1 }$, since $\breve \cW_{n} \onto \widetilde \L_{n-1 }$, we can pick lifts $\hat e_i \in \breve \cW_{n}$ of $e_i$.
Consider the $\gs$-module $\breve \W_{n}$ generated by $p \hat e_i, u ^{a_i} \hat e_i $ and $\alpha _j $; we claim this is a Kisin model of $\breve \cW_{n}$.

We see that $\hat e, \alpha$ generate $\breve \cW_{n}$ as \'etale $\varphi$-modules, so it suffices to check that $\breve \W_{n}$ is $\varphi$-stable.
To see this, we first observe that
$$\varphi (\hat e_1, \dots, \hat e_m) = (\hat e_1, \dots, \hat e_m) A + (\alpha_1, \dots , \alpha_t) B, $$ where $A \in \textnormal{M}_{m\times m}(\gs), B \in \textnormal{M}_{t\times m}(\O_{\mathcal E})$. Now we check step by step.
\begin{itemize}
  \item Because $p\alpha_i=0$, we have $\varphi (p\hat e_1, \dots, p\hat e_m) = (p\hat e_1, \dots, p\hat e_m) A $ with $A \in \textnormal{M}_{m\times m}(\gs)$.

  \item $\varphi (u ^{a_1} \hat e_1 , \dots , u ^{a_m}\hat e_m  )= (\hat e_1, \dots \hat e_m ) A \varphi (\Lambda) + (\alpha_1, \dots , \alpha_t) B \varphi (\Lambda) $ where $\Lambda$ is a the diagonal matrix $[u ^{a_1}, \dots, u ^{a_m}]$. Since $\{\alpha_i,  u ^{a_j}\bar e_j\}$ forms a $\ku$-basis of $\gZ_{n}$, we must have
      $$(e_1, \dots , e_m  )A \varphi(\Lambda) \equiv (e_1,\dots , e_m )\Lambda C\mod p$$
      with $C$ having coefficients in $\gs$ and $B \varphi (\Lambda)$ have entries in $\ku$. Hence $\{\varphi (u ^{a_i} e_i)\} \subset \breve \W_{n} $.

  \item Finally it is obvious that $\varphi (\alpha_1, \dots , \alpha_t) \subset  \N_n$.
\end{itemize}
So we conclude that $\breve \W_{n}$ is indeed a Kisin model of $\cW_{(n)}$, and the sequence \eqref{Eq-exactsq-gW-n} is short exact; as we noted in Step 0, this concludes the proof.
\end{proof}

We now proceed to Lemma \ref{lemma bound h}. First, let us introduce a useful definition.
\begin{defn}
Let $\M$ be a $p$-torsion Kisin module (over $k'[\![u]\!]$). Then $v_R (\det (\varphi |_\M))$ makes sense and does not depend on the choice of $k'[\![u]\!]$-basis of $\M$ (here we normalize $v_R $ on $k'[\![u]\!]$ by setting $v_R(u )=1$). Denote $\alpha(\M) = v_R (\det (\varphi |_\M))$.
\end{defn}

\begin{lemma}\label{lemma alpha}
\begin{enumerate}
  \item If we have an exact sequence of $p$-torsion Kisin modules $0 \to \M^{(1)} \to \M^{(2)}\to \M^{(3)} \to 0$, then $\alpha(\M^{(2)}) = \alpha(\M^{(1)}) + \alpha(\M^{(3)})$.

  \item  Let $L$ be a $G_{K'}$-stable $\Z_p$-lattice in a semi-stable representation $V$  with non-positive Hodge-Tate weights $\t{HT} (V)$, and $\L$ the corresponding Kisin module.
      Then $\alpha(\L/p\L) = e(\sum_{i \in \t{HT} (V)} -i)$ where $e=e(K')$ the ramification index.
\end{enumerate}
\end{lemma}

\begin{lemma}\label{lemma bound h}
As we claimed at the end of proof of Proposition \ref{prop-key}, there exists a constant $h$ only depending on $d = \dim_{\BF_p} T_1 $, $p$, $f$ and $e$ such that $a_i \leq h$.
\end{lemma}
\begin{proof}
Recall that $\mathring \L_{n-1}$ has a $k'[\![u ]\!]$-basis formed by $u^{a_1}\bar e_1, \ldots u^{a_m}\bar e_m$, where $\bar e_1, \ldots \bar e_m$ is a $k'[\![u ]\!]$-basis of $\L_{n-1}/p\L_{n-1}$.
We clearly have
$$\alpha(\mathring \L_{n-1}) = (p-1)\sum_{1 \le i\le m}a_i +\alpha (\widetilde \L_{n-1}/p \widetilde \L_{n-1})
=(p-1)\sum_{1 \le i\le m}b_i +\alpha \left(\widetilde \L_{n-1}(-s)/p \widetilde \L_{n-1}(-s)\right).$$
 Since $b_i \le \sum_{1 \le i\le m}b_i$, so we only need to bound $\alpha(\mathring \L_{n-1})-\alpha (\widetilde \L_{n-1}(-s)/p \widetilde \L_{n-1}(-s))$.

By Lemma \ref{lemma alpha}(1) and Equation \eqref{seq-1}, $\alpha(\mathring \L_{n-1}) \le \alpha(\mathfrak Z_n)$.
Recall from diagram \eqref{Diag-2 first apperance}, $\mathfrak Z_n$ is realized via $\widetilde Z_n \onto Z_n\otimes_{\Fp}k_E \onto Z_n$. That is, we have
$$\widetilde{\mathfrak Z_n} \onto \widetilde{\mathfrak Z_n}/\varpi_E \widetilde{\mathfrak Z_n} =\oplus\mathfrak Z_n \onto \mathfrak Z_n,$$
where $\mathfrak Z_n/\varpi_E \mathfrak Z_n$ is a direct sum of ($[k_E:\Fp]$)-copies of $\mathfrak Z_n$. So we have
$$\alpha(\mathring \L_{n-1}) \le \alpha(\mathfrak Z_n) \le \alpha(\widetilde{\mathfrak Z_n}/\varpi_E \widetilde{\mathfrak Z_n}) \le \alpha(\widetilde{\mathfrak Z_n}/p\widetilde{\mathfrak Z_n}).$$

So we have
\begin{eqnarray*}
 \alpha(\mathring \L_{n-1})-\alpha \left(\widetilde \L_{n-1}(-s)/p \widetilde \L_{n-1}(-s)\right)
 &\le & \alpha(\widetilde{\mathfrak Z_n}/p\widetilde{\mathfrak Z_n}) - \alpha \left(\widetilde \L_{n-1}(-s)/p \widetilde \L_{n-1}(-s)\right) \\
 &= & e ( \sum_{i \in \t{HT} (\widetilde N_n)} -i  ), \textnormal{ by Lemma \ref{lemma alpha}(2)}.
\end{eqnarray*}
Here, the representation $\widetilde N_n$ comes from the first row of \eqref{Diag-2 first apperance}.
Recall that by Lemma \ref{rmk: big E}, we could have chosen our $E$ to be $K_d$. By Corollary \ref{rem min HT wts}, we have $\sum_{i \in \t{HT} (\widetilde N_n)} -i \le efd^2 (p^{fd}-2)$ (since $\widetilde N_n [\frac 1 p]$ has $\Qp$-dimension $efd^2$). So we can let $h$ be an integer such that
 $$h \ge \frac{es}{p-1}+ e\frac{efd^2 (p^{fd}-2)}{p-1}.$$
By Equation \eqref{eq s}, we can let $s= p^{fd}+p-2$. So for example, we can let
\begin{equation} \label{eq h}
  h:= 3fe^2d^2p^{fd}.
\end{equation}
\end{proof}

\section{Overconvergent basis and main theorem} \label{sec OC}

In this section, we prove our main theorem, namely, the overconvergence of $(\varphi, \tau)$-modules associated to finite free $\Zp$-representations. To do so, we first show the existence of an ``overconvergent basis", with respect to which all entries of the matrices for $\varphi$ and $\hat G$ are overconvergent elements.

\subsection{Existence of overconvergent basis}
\begin{theorem} \label{thm OC basis}
Let $\rho : G_K \to \GL_d (\Z_p)$ be a continuous representation and $\hat M = (M, \varphi_M, \hat G)$ the associated $(\varphi, \tau)$-module.
Then there exists an $\O_\E$-basis of $M$, and a constant $\alpha = \alpha (p, f, e , d)$ that only depends on $p, e, f $ and $d$, such that with respect to this basis,
\begin{itemize}
  \item the matrix of $\varphi_M$ is in $\Md(\gs [\![\frac{p}{u ^\alpha}]\!])$,
  \item the matrix of $\tau$ is in $\Md(W(R)[\![\frac{p}{u ^\alpha} ]\!])$ for any $\tau \in \hat G$.
\end{itemize}
\end{theorem}

Before we proceed to the proof, we will adopt the following convention on notations.
\begin{convention}
To ease the notation, we will use notations like $(e_j)$ to mean a row vector $(e_j)_{j=1}^d=(e_1, \dots, e_d)$.
\end{convention}

\begin{proof}
For the proof, we first inductively construct a specific set of generators for $\m_{(n)}$. Then we use them to define a basis for $M_n$. We show that the bases are compatible, and gives us the desired overconvergent basis by taking inverse limit .

\textbf{Step 1.} \emph{Generators of $\m_{(n)}$}.
First of all, by induction on $n$, we construct a specific set of generators $\{\e_{(n), j}^{(i)},  1 \le i \le n\}_{j=1}^d$ of $\m_{(n)}$ such that $\{p^{i-1}\e_{(n), j}^{(i)}\}_{j=1}^d$ forms a $\ku$-basis of $\m_{(n)}^{i-1, i}$.

We choose $\{\e_{(1), j}^{(1)}\}_{j=1}^d$ any $\ku$-basis of $\m_{(1)}$.
Suppose we have defined $\{\e_{(n-1), j}^{(i)},  1 \le i \le n-1\}_{j=1}^d$ the generators for $\m_{(n-1)}$.
Recall that in Corollary \ref{cor identify}, we have used the map $\iota_{n-1, n}: M_{n-1} \to M_{n}$ to identify $\m_{(n-1)}^{i-1, i}$ with $\m_{(n)}^{i-1, i}$ when $1 \le i \le n-1$. Now define
\begin{itemize}
  \item $\e_{(n), j}^{(i)} : =\iota_{n-1, n} (\e_{(n-1), j}^{(i)}) $ for $1 \le i \le n-1, 1\le j \le d$, and
  \item choose any $\{\e_{(n), j}^{(n)}\}_{j=1}^d$ in $\m_{(n)}$, so that $(p^{n-1}\e_{(n), j}^{(n)})$ is
  a $\ku$-basis of $\m_{(n)}^{n-1, n}=p^{n-1}\m_{(n)}$.
\end{itemize}
This finishes the inductive definition.

\textbf{Step 2.} \emph{Basis of $M_n$.}
With above, now we define basis for $M_n$.
By Lemma \ref{lem-devissage}(2), for any $x\ge 1$, we can write
\begin{eqnarray} \label{eq1}
(p\e_{(x), j}^{(x)}) =(\iota_{x-1, x}(\e_{(x-1), j}^{(x-1)}))  \Lambda_{x-1} (I_d+ \frac{p}{u^{2h}} Z_{x-1})
\end{eqnarray}
with $Z_{x-1} \in \Md(\gs[\frac{p}{u^{2h}}])$.
Now define
\begin{eqnarray}\label{eq2}
(e^{(n)}_j)  : = (\e_{(n), j}^{(n)})  \left(\prod_{x=1}^{n-1}  (  \Lambda_{x-1} ( I_d +  \frac{p}{u^{2h}} Z_{x-1}  )     )\right)^{-1}  .
\end{eqnarray}
We denote $Y_n:=\prod_{x=1}^{n-1}  (  \Lambda_{x-1} ( I_d +  \frac{p}{u^{2h}} Z_{x-1}  ) )$, and so $(e^{(n)}_j)   =  (\e_{(n), j}^{(n)}) Y_n^{-1}$.
Since $(\e_{(n), j}^{(n)})$ are killed by $p^n$, we can take $Y_n$ and  $Y_n^{-1}$ to be in $\Md(\mathcal O_{\mathcal E, n})$, and we do so. More precisely, we can arrange that $Y_n \in \Md(\frac{1}{u^{2(n-1)h}}\gs)$, and $Y_n^{-1} \in \Md(\frac{1}{u^{3(n-1)h}}\gs)$; these are only rough estimates, but they are sufficient for our use.

Since $\e_{(n), j}^{(n)}, 1\le j \le d$ lifts an $\mathcal O_{\mathcal E, 1}$-basis of $M_1$, by Nakayama Lemma, $ (e^{(n)}_1, \ldots e^{(n)}_d)$ is an $\mathcal O_{\mathcal E, n}$-basis of $M_n$.

\textbf{Step 3.} \emph{Compatibility of basis.}
We check that $(e^{(n-1)}_j) = (e^{(n)}_j  \mod p^{n-1})$.
Note that the composite of the map $M_n \onto M_{n-1} \hookrightarrow M_n$ is precisely the multiplication by $p$ map, where the first map is modulo $p^{n-1}$ and the second map is $\iota_{n-1, n}$.
So it suffices to check that $\iota_{n-1, n}(e^{(n-1)}_j)  = p(e^{(n)}_j)$. But this is easy consequence of \eqref{eq1} and \eqref{eq2}.

So now we can define $(e_j) : =\lim_{n \to \infty} (e^{(n)}_j)$, which is a basis of $M$.

\textbf{Step 4.} \emph{Matrices for $\varphi$ and $\tau$.}
Let $\tau$ be any element in $\hat G$, we claim that we have
\begin{itemize}
  \item $\varphi(e_j) =(e_j) A$ with $A\in \Md(\gs [\![\frac{p}{u^{\alpha}}]\!]).$

\item  $\tau(e_j) =(e_j) B$ with $\varphi (B) \in \Md(\hR [\![\frac{p}{u^{\alpha}}]\!]).$
\end{itemize}
By Lemma \ref{lem shrink ring}, it suffices to prove that
\begin{itemize}
  \item $\varphi(e_j^{(n)}) =(e_j^{(n)}) A_n$ with $A_n\in \Md(\frac{1}{u^{(n-1)\alpha}}\gs)$.

  \item $\tau(e_j^{(n)}) =(e_j^{(n)}) B_n$ with $\varphi (B_n) \in \Md(\frac{1}{u^{(n-1)\alpha}}\hR)$.
\end{itemize}
Since $\m_{(n)}$ comes from a loose crystalline lift, we can write
\begin{itemize}
  \item $\varphi(\e_{(n), j}^{(n)}) = \sum_{i=1}^n (\e_{(n), j}^{(i)}) P_i^{(1)}$, with $P_i^{(1)} \in \Md(\gs)$,
  \item $\tau(1\otimesvarphi \e_{(n), j}^{(n)}) = \sum_{i=1}^n (1\otimesvarphi \e_{(n), j}^{(i)}) Q_i^{(1)}$, with $Q_i^{(1)} \in \Md(\hR)$.
\end{itemize}
By Lemma \ref{lem-devissage}, for all $i<n$, we can write
$$ (\e_{(n), j}^{(i)}) =(\e_{(n), j}^{(n)}) Y_{i, n}  \text{ with }  Y_{i, n} \in \Md(\gs[\frac{p}{u^{2h}}]).$$
So we can write
\begin{itemize}
  \item $\varphi(\e_{(n), j}^{(n)}) =(\e_{(n), j}^{(n)}) P_n$, with $P_n \in \Md(\gs[\frac{p}{u^{2h}}])$,
  \item $\tau(1\otimesvarphi \e_{(n), j}^{(n)}) =(1\otimesvarphi  \e_{(n), j}^{(n)}) Q_n   $, with $Q_n \in \Md(\hR[\frac{p}{u^{2h}}])$.
\end{itemize}
So
\begin{itemize}
  \item $\varphi( e^{(n)}_j )   = ( e^{(n)}_j ) Y_n P_n \varphi(Y_n^{-1})$,
  \item $\tau(1\otimesvarphi e^{(n)}_j )   = ( 1\otimesvarphi e^{(n)}_j ) Y_n Q_n \tau(Y_n^{-1}),
$.
\end{itemize}
and so (using the fact $e^{(n)}_j$ is killed by $p^n$  and then $(\frac{p}{u ^{2h}})^m $ will vanishes for $m \geq n$)
\begin{itemize}
  \item $A_n=  Y_n P_n \varphi(Y_n^{-1}) \in \Md(\frac{1}{u^{2(n-1)h+2(n-1)h+3(n-1)ph}}\gs),$
  \item $\varphi(B_n)= \varphi(Y_n) Q_n \tau(\varphi(Y_n^{-1})) \in \Md(\frac{1}{u^{2(n-1)ph+2(n-1)h+3(n-1)ph}}\hR).$
\end{itemize}
Finally, we can simply let
\begin{equation}\label{eq alpha}
\alpha=7ph
\end{equation}
to conclude.
Note that since $\varphi (B) \in \Md(\hR [\![\frac{p}{u^{\alpha}}]\!])$, we have $B \in \Md(W(R)[\![\frac{p}{u^{\alpha/p}}]\!])$.
\end{proof}

\begin{lemma} \label{lem shrink ring}
Let $B$ be the ring $\gs$, or $\hR$, or $W(R)$.
Suppose that $y_n = \frac{x_n}{u^{\alpha (n-1)}} \in W_n(\Fr R)$ such that  $y_{n +1} \equiv y _{n}  \mod p ^n $ in $W_{n+1} (\Fr R)$.
If $x_n \in B$ for all $n$, then $y_n$ converges to a $y \in B[\![\frac{p}{u ^\alpha}]\!]$.
\end{lemma}
\begin{proof}
It suffices to show that $y_n \in B_n[\frac{p}{u ^\alpha}]$ for all $n$. This is true when $n=1$. Suppose it is valid for $n = m$. Let us consider the case that $n = m+1$.
Since $y_m \in B_{m}[\frac{p}{u ^\alpha}]$, we can lift it to an element in $ B_{m+1}[\frac{p}{u ^\alpha}]$ and still denote it by $y_m$, then
$x_{m}= u^{(m-1)\alpha} y_{m} \in B_{m+1}$.
Now
$$ y_{m+1} = y_m +(y_{m+1}-y_m)= y_m +  \frac { x_{m+1} - u^\alpha x_{m}}{u ^{\alpha m}}.$$
Since $y_{m+1} \equiv y _{m} \mod p ^m$, we have $x_{m+1} - u ^\alpha x_m \in p ^m W(R)_{m+1}\cap B_{m+1} =p^m B_{m+1}$, and so $x_{m+1} - u ^\alpha x_m = p ^m x'_{m +1}$ with $x'_{m+1} \in B_{m+1}$. So $y_{m+1} = y_m + \frac{p^m}{ u ^{\alpha m}} x'_{m+1}$. This completes the induction and proves the lemma.
\end{proof}

\subsection{Overconvergence}
Finally, we can prove our main theorem.
\begin{thm} \label{thm final OC}
Let $T$ be a continuous finite free $\Zp$-representation of $G_K$, and let $\hat M = (M, \varphi_M, \hat G)$ the associated $(\varphi, \tau)$-module. Then $\hat M$ is overconvergent.
\end{thm}

\begin{proof}
Let $(e_1, \dots, e_d)$ be the basis of $M$ as in Theorem \ref{thm OC basis}. Let $(t_1, \dots, t_d)$ be any basis of $T$, and let $(e_1, \dots, e_d)=(t_1, \dots, t_d)X$ where $X \in \Md(\O_{\widehat \E^\ur})$. Then we have $\varphi(X)=XA$.
In order to prove the theorem, it suffices to show that $X \in \Md(W(\Fr R)^{\dagger, r})$ for some $r>0$. We claim that $X \in \Md(W(R)[\![\frac{p}{u^\alpha}]\!])$.

To prove the claim, by Lemma \ref{lem shrink ring}, it suffices to show that $u^{(n-1)\alpha}X_n \in \Md(W_n(R))$ where $X_n:=X(\mod p^n)$. We prove this by induction on $n$.
When $n=1$, we have $\varphi(X_1)=X_1A_1$. Since $A_1 \in \Md(\gs_1)$ by Theorem \ref{thm OC basis}, we have $X_1 \in \Md(R)$ by Lemma \ref{lem W(R) eqn}.
Suppose the claim is true when $n \leq m$, and let us consider the case $n=m+1$.
We can always write $X = \sum\limits_{\ell = 0} ^\infty p ^\ell X'_\ell$ with $ X'_\ell \in \Md([\Fr R])$, where we use $[\Fr R]$ to mean the set of Teichm\"{u}ller lifts.
We have $X_{m+1}= \sum\limits_{\ell = 0} ^{m}p ^\ell X'_\ell$.
To finish the induction process, it suffices to show $u^{m\alpha} X'_m \in \Md([R])$.
From $\varphi(X_{m+1})=X_{m+1}A_{m+1}$, we can deduce
\begin{eqnarray*}
p^m\varphi(u^{m\alpha} X'_m) = p^m u^{m\alpha} X'_m u^{(p-1)m\alpha} A_{m+1} + u^{pm\alpha} X_mA_{m+1} -\varphi(u^{m\alpha}X_m)
\end{eqnarray*}
By induction hypothesis, $u^{(m-1)\alpha} X_m \in \Md(W(R))$, and by Theorem \ref{thm OC basis}, $u^{m\alpha} A_{m+1} \in \Md(W(R))$, so the above equation is reduced to the form
$$p^m\varphi(u^{m\alpha} X'_m)- p^m u^{m\alpha} X'_m u^{(p-1)m\alpha}  A_{m+1}  = C, $$
for some $C \in \Md(W(R))$. Since $p^mW(\Fr R)\cap W(R)=p^m W(R)$, we have $C \in \Md(p^m W(R))$. Divide both sides of the equation by $p^m$ and modulo $p$ on both sides, then we can apply Lemma \ref{lem W(R) eqn} to conclude that $u^{m\alpha} X'_m \in \Md([R])$. This completes the proof of our theorem.

As a final note, it is easy to see that the overconvergence radius $r$ can be any number $< \frac 1 \alpha$. Combined with Equation \eqref{eq alpha} and Equation \eqref{eq h}, it suffices that
\begin{equation} \label{eq OC radius}
  r < \frac{1}{21pfe^2d^2p^{fd}}.
\end{equation}
\end{proof}

\begin{lemma}\label{lem W(R) eqn}
Suppose $Y \in \Md(\Fr R)$ and $B, C \in \Md(R)$ such that $\varphi(Y)= YB+C$, then we must have $Y \in \Md(R)$.
\end{lemma}
\begin{proof}
Suppose otherwise, consider an entry $y_{i, j}$ in the matrix $Y$ with minimal $v_R$ valuation, then one can easily deduce a contradiction.
\end{proof}

\end{document}